\documentclass[reqno,11pt]{amsart}

\usepackage[a4paper,left=25mm,right=25mm,top=30mm,bottom=30mm,marginpar=20mm]{geometry} 
\usepackage{amsmath}
\usepackage{amssymb}
\usepackage{amsthm}
\usepackage{amscd}
\usepackage[ansinew]{inputenc}
\usepackage[final]{graphicx}
\usepackage{tikz}
\usepackage{bbm}

\renewcommand{\epsilon}{\varepsilon}

\numberwithin{equation}{section}

\newtheoremstyle{thmlemcorr}{10pt}{10pt}{\itshape}{}{\bfseries}{.}{10pt}{{\thmname{#1}\thmnumber{ #2}\thmnote{ (#3)}}}
\newtheoremstyle{thmlemcorr*}{10pt}{10pt}{\itshape}{}{\bfseries}{.}\newline{{\thmname{#1}\thmnumber{ #2}\thmnote{ (#3)}}}
\newtheoremstyle{defi}{10pt}{10pt}{\itshape}{}{\bfseries}{.}{10pt}{{\thmname{#1}\thmnumber{ #2}\thmnote{ (#3)}}}
\newtheoremstyle{remexample}{10pt}{10pt}{}{}{\bfseries}{.}{10pt}{{\thmname{#1}\thmnumber{ #2}\thmnote{ (#3)}}}
\newtheoremstyle{ass}{10pt}{10pt}{}{}{\bfseries}{.}{10pt}{{\thmname{#1}\thmnumber{ A#2}\thmnote{ (#3)}}}

\theoremstyle{thmlemcorr}
\newtheorem{theorem}{Theorem}
\numberwithin{theorem}{section}
\newtheorem{lemma}[theorem]{Lemma}
\newtheorem{corollary}[theorem]{Corollary}
\newtheorem{proposition}[theorem]{Proposition}

\theoremstyle{thmlemcorr*}
\newtheorem{theorem*}{Theorem}
\newtheorem{lemma*}[theorem]{Lemma}
\newtheorem{corollary*}[theorem]{Corollary}
\newtheorem{proposition*}[theorem]{Proposition}
\newtheorem{problem*}[theorem]{Problem}
\newtheorem{conjecture*}[theorem]{Conjecture}

\theoremstyle{defi}

\theoremstyle{remexample}
\newtheorem{remark}[theorem]{Remark}
\newtheorem{example}[theorem]{Example}

\theoremstyle{ass}

\newcommand{\Acal}{\mathcal{A}}

\newcommand{\Ecal}{\mathcal{E}}

\newcommand{\Ical}{\mathcal{I}}
\newcommand{\Jcal}{\mathcal{J}}

\newcommand{\Lcal}{\mathcal{L}}

\newcommand{\Ocal}{\mathcal{O}}

\newcommand{\Vcal}{\mathcal{V}}

\DeclareMathOperator{\diverg}{div}

\DeclareMathOperator{\dist}{dist}

\newcommand{\norm}[1]{\|#1\|}

\newcommand{\sprlr}[1]{\left( #1 \right)}

\newcommand{\dd}{\;\mathrm{d}}

\newcommand{\N}{\mathbb{N}}
\newcommand{\R}{\mathbb{R}}

\newcommand{\sym}{\mathrm{sym}}
\newcommand{\skw}{\mathrm{skew}}

\newcommand{\weakly}{\rightharpoonup}

\newcommand{\eps}{\epsilon}

\newcommand{\ffi}{\varphi}

\newcommand{\Qld}{Q^\ast}

\DeclareMathOperator{\SO}{SO}

\def\XXint#1#2#3{{\setbox0=\hbox{$#1{#2#3}{\int}$} 
\vcenter{\hbox{$#2#3$}}\kern-.5\wd0}}

\DeclareMathOperator{\cof}{cof}
\DeclareMathOperator{\tr}{Tr}

\DeclareMathOperator{\Glim}{\Gamma\text{-}\lim}
\DeclareMathOperator{\Id}{Id}
\AtBeginDocument{\let\div\relax\DeclareMathOperator{\div}{div}}
\newcommand{\rn}[1]{\nabla^\epsilon #1_\epsilon}
\newcommand{\ui}[1]{^{\left(#1\right)}}

\newcommand{\tnabla}{{\widetilde \nabla}}
\newcommand\restrict[1]{\raisebox{-.5ex}{$|$}_{#1}}

\title[Theories for incompressible rods]{Theories for incompressible rods: \\a rigorous derivation via $\Gamma$-convergence}

\author{Dominik Engl}
\address{Mathematisch Instituut, Universiteit Utrecht, Postbus 80010, 3508 TA Utrecht, The Netherlands}
\email{D.M.Engl@uu.nl}

\author{Carolin Kreisbeck}
\address{Mathematisch Instituut, Universiteit Utrecht, Postbus 80010, 3508 TA Utrecht, The Netherlands}
\email{C.Kreisbeck@uu.nl}

\begin{document}

\maketitle
 \begin{abstract}  
 \vspace{-12pt} 
We use variational convergence to derive a hierarchy of one-dimensional rod theories, starting out from three-dimensional models in nonlinear elasticity subject to local volume-preservation. 
The densities of the resulting $\Gamma$-limits are determined by minimization problems 
with a trace constraint that arises from the linearization of the determinant condition of incompressibility. 
While the proofs of the lower bounds rely on suitable constraint regularization, the upper bounds require a careful, explicit construction of locally volume-preserving recovery sequences. After decoupling the cross-section variables with the help of divergence-free extensions, we apply an inner perturbation argument to enforce the desired non-convex determinant constraint. 
To illustrate our findings, we discuss the special case of isotropic materials.  
\vspace{8pt}

 

 \noindent\textsc{Date:} \today.
 \end{abstract}


\section{Introduction}

The study of the deformation behavior of thin structures in response to external forces dates back centuries, with pioneering contributions on the bending of elastic rods by Euler and Bernoulli, and the formulation of a plate theory by Kirchhoff.
And still, nowadays, the topic has not lost any of its relevance, as numerous modern applications in technology and
recent developments in the life sciences demonstrate. Thinking, for instance, of carbon nanowires and 
printed electronics in computer devices, of fiber-reinforced and layered composites in materials science, or of
cell membranes and  DNA strands 
in biology, current research directions require a profound understanding of
elastic bodies with a small extension in one or two spatial dimensions.

Whereas classical approaches based on asymptotic expansions had been remarkably fruitful in the small-strain setting of linear elasticity, see e.g.~\cite{Cia97,TrV96}, accounting for large deformations calls for a mathematical framework that is well-suited to deal with geometric nonlinearity. 
In a variational convext, dimension reduction via $\Gamma$-convergence \cite{Bra02, Dal93} allows to establish a  rigorous connection between fully three-dimensional models in hyperelasticity and lower-dimensional theories for thin structures.

The first results in this spirit are
due to Acerbi, Buttazzo \& Percivale \cite{ABP91}, who proved a $3$d-$1$d reduction for elastic strings, and to Le Dret \& Raoult \cite{LeR95,LeR96}, who deduced a model for two-dimensional elastic membranes from the $\Gamma$-limit of elastic energy functionals for vanishing thickness. 
A few years later, an ansatz-free derivation of Kirchhoff's plate theory was obtained independently in~\cite{FJM02, Pan03}. These  seminal findings, in particular, the quantitative geometric rigidity estimate by Friesecke, James \& M\"uller~\cite{FJM02}, actuated substantial efforts towards a systematic analysis of different types of thin structures, with contributions by many authors. We highlight here a few selected examples. Considering that the scaling of the acting external forces has a decisive influence on the resulting lower-dimensional models, a complete hierarchy of plate models was  derived in~\cite{FJM06};  
details about the asymptotic analysis one-dimensional objects, precisely strings and rods, in various scaling regimes can be found in~\cite{MoM03, MoM04, Sca09}. Recently, 
various special features of thin structures have been investigated, including small-scale heterogeneities in plates and rods, which require a combination of dimension reduction and homogenization techniques~\cite{HNV14, Neu12, NeV13}, global invertibility aiming to avoid self-interpenetration of matter~\cite{Bar19, Hor11, OlR17}, or thin objects made of materials with pre-existing strain~\cite{BLS16, CRS17b, CRS17a, KoO18}. 
While the previously mentioned works rely on $\Gamma$-convergence, and hence (assuming suitable compactness) imply that (almost) minimizers converge, we refer e.g.~to~\cite{BPV17, DaM12, MoM08, MoS12, MuP08} for statements on the convergence of equilibria.

An important class of thin structures are those made of incompressible materials, which are commonly used to describe rubber-like substances \cite{ACD07, DeD02, Ogd72}, and thus, occur e.g.~in blood vessels, tires, seat belts, etc.  
From an analytical point of view, there is recent work on membranes~\cite{CoD06}, plates in the Kirchhoff \cite{CoD09} and von K\'arm\'an regime  \cite{ChL13,LeL15}, hyperelastic shells~\cite{ABR19}, and strings~\cite{EnK19}. 
The challenge in the mathematical analysis of these models lies in the non-convex constraint imposed on the elastic energy functionals to guarantee local volume-preservation; precisely, the Jacobian determinant of admissible deformation fields has to be constant and equal to one, cf.~\cite{CoD15}.    

Our intention with this article is to close a gap in the literature by deriving a hierarchy of theories for incompressible rods, including, in particular, the Kirchhoff- and von K\'arm\'an-type cases.
To this end, we characterize the asymptotic behavior  in the limit of vanishing cross-section of suitably rescaled elastic energy functionals subject to a local volume-preservation constraint. 

The paper is organized as follows: In the remainder of the introduction, we give the precise problem formulation, announce the main results, give insight into our methodological approach, and specify relevant notation. 
As preliminaries, we discuss in Section~\ref{sec:densities} several properties of the limit densities arising through the dimension reduction procedures, and Section~\ref{sec:tools} 
collects the most important technical tools for proving the upper bounds. 
The core of this work, Section~\ref{alpha=2}, contains the main $\Gamma$-convergence result for elastic rods in the Kirchhoff regime under the assumption of incompressibility. We identify the reduced $\Gamma$-limit for shrinking cross-section, determine the corresponding Euler-Lagrange equations, and specify our findings for the case of isotropic materials. 
Finally, the asymptotic analysis for the three remaining scaling regimes is presented in Section~\ref{alpha>2}. 

\subsection{Problem formulation} 
Throughout the paper, let $L>0$ and $\omega\subset \R^2$ be a bounded, simply connected Lipschitz domain of unit measure,
i.e., $\Lcal^2(\omega) = |\omega| = \int_\omega \dd \tilde x=1$, such that
\begin{align}\label{centerofmass}
	\int_\omega x_2 \dd \tilde x = \int_\omega x_3 \dd \tilde x = \int_\omega x_2x_3 \dd \tilde x = 0;
\end{align}
see~Section~\ref{sec:notation}, in particular,~\eqref{xtilde}, for the use of notation.

For small $\eps>0$, we introduce $\Omega_\eps = (0,L)\times \eps \omega \subset \R^3$ as the reference configuration of a thin, incompressible body of length $L$ with cross-section $\eps\omega$, and define its total energy through the functional $\Ecal_\eps : H^1(\Omega_\eps;\R^3) \to \R_\infty:=\R\cup\{\infty\}$ with
\begin{align}\label{Eeps}
	\Ecal_\eps(v)= \int_{\Omega_\eps} W(\nabla v) \dd y - \int_{\Omega_\eps} f_\eps\cdot v \dd y, \qquad v\in H^1(\Omega_\eps;\R^3);
\end{align}
here, $W: \R^{3\times 3} \to [0,\infty]$ is the constrained elastic energy density given by 
\begin{align}\label{def_W}
	W(F) = \begin{cases} W_0(F) &\text{ for } \det F = 1,\\ \infty &\text{ otherwise,} \end{cases}
\end{align}
where $W_0: \R^{3\times 3} \to [0,\infty)$ satisfies the following hypotheses:
\begin{itemize}
	\item[(H1)] $W_0$ is twice continuously differentiable in a neighborhood of $\SO(3)$;
	\item[(H2)] $W_0(\Id)=0$ and there is $C_1>0$ such that $W_0(F) \geq C_1\dist^2(F,\SO(3))$ for all $F\in \R^{3\times 3}$;
	\item[(H3)] $W_0$ is frame indifferent, i.e., $W_0(RF) = W_0(F)$ for all $F\in\R^{3\times 3}$ and $R\in\SO(3)$.
\end{itemize}

The vector field $f_\eps\in L^2(\Omega_\eps;\R^3)$ in~\eqref{Eeps} describes the external forces acting on the body. For the sake of simplicity, $f_\eps$ is assumed to be independent of the cross-section variables, and can therefore be interpreted as an element of $L^2(0,L;\R^3)$. Regarding the scaling properties of $f_\eps$, we suppose the existence of an $\alpha\geq 0$ and a suitable $f\in L^2(0,L;\R^3)$ such that 
\begin{align}\label{alpha}
f_\eps = \epsilon^\alpha f;
\end{align} 
we suppose that the body forces
average out to zero, i.e.,
\begin{align}\label{hyp1_f}
	\int_0^L f \dd x_1= 0.
\end{align}

With the intended asymptotic analysis of the energies in~\eqref{Eeps} in mind, it is technically convenient to perform a change of variables that allows us to replace $\Ecal_\eps$ with functionals defined on a fixed, parameter-independent space.  
Indeed, with $y=(x_1,\eps x_2, \eps x_3)$ for $x=(x_1,x_2,x_3)\in \Omega:=\Omega_1$, and $u\in H^1(\Omega;\R^3)$ given by
$u(x) = v(y)$ for $v\in H^1(\Omega_\eps;\R^3)$, the normalization of $\Ecal_\eps$ per unit volume turns into $\Jcal_\eps: H^1(\Omega;\R^3)\to \R_\infty$,
\begin{align*}
	\Jcal_\eps(u)= \int_\Omega W(\nabla^\eps u) \dd x -  \int_\Omega f_\eps \cdot u \dd x, \qquad u\in H^1(\Omega;\R^3),
\end{align*}
with the rescaled deformation gradient 
\begin{align*}
\nabla^\eps u = (\partial_1 u| \tfrac{1}{\eps} \partial_2 u| \tfrac{1}{\epsilon} \partial_3 u).
\end{align*}

It is well-known that the scaling behavior of $\Jcal_\eps$ depends on the parameter-dependence of the external forces; with $\alpha$ as in~\eqref{alpha}, one has that
\begin{align*}
\Jcal_\eps \sim \begin{cases} \alpha & \text{if $\alpha\in [0,2)$,}\\ 
2\alpha-2 & \text{if $\alpha\geq 2$.}\end{cases}
\end{align*} 
see~\cite{FJM06} for more details.

The scaling regimes $\alpha\in[0,2)$, which give rise to (degenerate) models for incompressible strings, are studied in~\cite{EnK19}. In this paper, we focus on the cases $\alpha\geq 2$ to deduce a hierarchy of theories for incompressible rods. Hence, the relevant functionals for us to work with are the rescaled energies
\begin{align}\label{Jcal_eps^alpha}
	\Jcal_\epsilon\ui{\alpha} : H^1(\Omega;\R^3)\to \R_\infty,\quad u\mapsto \frac{1}{\eps^{2\alpha-2}}\int_\Omega W(\nabla^\eps u) \dd x - \frac{1}{\epsilon^{\alpha-2}}\int_\Omega f \cdot u\dd x,
\end{align}
or, if we intend to consider exclusively the elastic energy contribution, 
\begin{align}\label{Ical_eps^alpha}
	\Ical_\epsilon\ui{\alpha} : H^1(\Omega;\R^3)\to [0,\infty],\quad u\mapsto \frac{1}{\eps^{2\alpha-2}}\int_\Omega W(\nabla^\eps u) \dd x.
\end{align}
The main results of this paper are characterizations of the $\Gamma$-limits $\Ical\ui{\alpha}$ 
of the sequences of energy functionals $(\Ical^\alpha_\epsilon)_{\epsilon}$
for all $\alpha \geq 2$ (see Theorem~\ref{theo:rods=2} and Theorem~\ref{theo:rods>2}); the explicit formulas of $\Ical^{(\alpha)}$ in the four qualitatively different regimes $\alpha=2$, $\alpha\in (2,3)$, $\alpha=3$, and $\alpha>3$ can be found in~\eqref{I^2},~\eqref{I^2-3} and~\eqref{I^3}, respectively. 

\subsection{Approach and techniques}\label{subsec:approach}
We adopt and tailor the methodology from \cite{CoD09} for incompressible plates to our situation of $3$d-$1$d dimension reduction. In doing so, it is essential to exploit the available results in \cite{MoM03, MoM04, Sca09} on the asymptotic analysis of compressible rods in the various scaling regimes.
In the following, we give a brief overview of the ideas behind the three steps for proving $\Gamma$-convergence, namely, compactness, lower and upper bound. For a comprehensive introduction to variational convergence and its properties, see e.g.~\cite{Bra02, Dal93}.

All compactness properties emerge as an immediate consequence of the literature on the respective compressible cases, considering that $W_0\leq W$, cf.~\eqref{def_W}. 

The key ingredient for the lower bound is a suitable approximation of the determinant constraint with the help of a suitable penalization term.  To be more precise, we consider for each $k\in\N$ a penalized energy density $W_k:\R^{3\times 3}\to [0, \infty)$ given by
\begin{align}\label{penalty_density}
	W_k(F) =  W_0(F) + \frac{k}{2}(\det F - 1)^2, \qquad F\in \R^{3\times 3};
\end{align}
clearly, the sequence $(W_k)_k$ is increasing, and converges pointwise to $W$ for $k\to \infty$. 
Since each $W_k$ meets the requirement for densities in the limit theory of compressible rods, the $\Gamma$-limits of the functionals $(\Ical\ui{\alpha}_{k,\eps})_{k,\eps}$
with 
\begin{align}\label{penalty_energy}
	\Ical\ui{\alpha}_{k,\eps}(u)= \frac{1}{\eps^{2\alpha-2}} \int_\Omega W_k(\nabla^\eps u) \dd x, \qquad u\in H^1(\Omega;\R^3),
\end{align}
are well-established; let us call the above-mentioned $\Gamma$-limits $\Ical\ui{\alpha}_k$. Showing that the pointwise limit of these $\Ical_k\ui{\alpha}$ for $k\to \infty$ is exactly $\Ical\ui{\alpha}$ yields the desired liminf inequality; our proof is based on the monotonicity and pointwise convergence of corresponding limit densities, which are in general defined only implicitly via infinite-dimensional minimization problems, cf.~Corollary~\ref{cor:convergence_Q}.

The upper bounds require the construction of energetically optimal approximating sequences of locally volume-preserving deformations for any admissible limit state. As a starting point, we take a sequence $(y_\eps)_\eps$ inspired by the recovery sequences from the unconstrained settings for the finite-valued density $W_0$ (under consideration of the respective scaling regime). 
These recovery sequences typically involve higher-order terms in $\eps$ that are determined by the solutions to the variational problems arising in the definition of the limit densities.
In our incompressible setting, the corresponding minimization problem features a trace constraint - connected with the determinant constraint through linearization - from which we can deduce that
\begin{align}\label{det_intro}
\det \nabla^\eps y_\eps - 1 \sim \eps^\gamma
\end{align}
for some $\gamma>0$ depending on the scaling regime $\alpha$.

With~\eqref{det_intro} at hand, an inner perturbation argument in the spirit of~\cite{CoD06}, adapted for $3$d-$1$d reductions in \cite[Lemma 2.3]{EnK19} (see Lemma~\ref{lem:reparam_det=1} below), allows us to replace $(y_\eps)_\eps$ by a sequence that satisfies the incompressibility constraint exactly. 
We remark that this auxiliary result is only applicable if the cross-section variables of $y_\eps$ are decoupled, for which 
we extend $y_\eps$ suitably to a cuboid containing $\Omega$. Technically, this task reduces to finding a divergence-free extension (see e.g. \cite{KMPT00}) in the cross-section direction.

Overall, our analysis shows that the following diagram commutes:
\begin{figure}[h!]
	\centering
	\begin{tikzpicture}[scale=0.75]
		\begin{scope}
			\draw [->] (1,0) -- (4,0);
			\draw [->] (0,-1) -- (0,-4);
			\draw [->] (1,-5) -- (4,-5);
			\draw [->] (5,-1) -- (5,-4);
			\draw (0,0) node {$\Ical_{k,\epsilon}\ui{\alpha}$};
			\draw (5,0) node {$\Ical_\epsilon\ui{\alpha}$};
			\draw (0,-5) node {$\Ical_k\ui{\alpha}$};
			\draw (5,-5) node {$\Ical\ui{\alpha}$};
			\draw (2.5,0) [anchor = south] node {$k\to \infty$};
			\draw (2.5,-5) [anchor = north] node {$k\to \infty$};
			\draw (0,-2.25) [anchor = east] node {$\Glim_{\epsilon\to 0}$};
			\draw (5,-2.25) [anchor = west] node {$\Glim_{\epsilon\to 0}$};
		\end{scope}
	\end{tikzpicture}
\end{figure}

\subsection{Notation}\label{sec:notation}
	The following notations are used throughout the paper. 
	Let $e_1,e_2,e_3$ be the standard unit basis vectors in $\R^3$, and let $a^\perp:=(-a_2, a_1)$ for $a\in \R^2$. For any two vectors $a, b\in \R^n$, we denote by $a\cdot b$ their standard inner product and by $a\otimes b \in \R^{n\times n}$ their tensor product, that is, componentwise, $(a\otimes b)_{ij} = a_ib_j$ for $i,j=1,\ldots, n$. 
	The inner product on the space of matrices $\R^{m\times n}$ is given by $A:B= \tr(AB^T)$ for $A,B\in\R^{m\times n}$, where $\tr$ is the trace operator and $B^T$ the transpose of $B$. The induced norms on $\R^n$ and $\R^{m\times n}$ are both denoted by $|\cdot|$.
	Moreover, $A^{\sym} = \frac{1}{2}(A^T+A)$ refers to the symmetric part of $A\in \R^{n\times n}$, we use $\Id$ for the identity matrix in $\R^{n\times n}$, $\SO(n)\subset \R^{n\times n}$ is the rotation group, and $\R^{n\times n}_{\skw}$ stands for the space of skew-symmetric $n\times n$ matrices. 
	If $g:\R^{3\times 3}\to \R$, $b\in\R^3$ and $A\in \R^{3\times 2}$, we simplify the expression $g((b|A))$ to $g(b|A)$.

	Furthermore, for $x=(x_1,x_2,x_3)\in\R^3$, we write $x=(x_1,\tilde x)$ with 
	\begin{align}\label{xtilde}
	\tilde x = (x_2,x_3)\in\R^2;
	\end{align} 
 in particular, the points of any subset of $U\subset \R^2$ are addressed by $\tilde x\in U$.
	Likewise, we split the components of $\R^3$-valued maps, that is $w=(w_1, \tilde w)$ for $w:U\subset \R^m\to \R^3$. 
	We denote the partial derivative of a function $w: U\subset \R^3 \to \R^m$ with respect to $x_i$ for $i=1,2,3$ by $\partial_i w$. 
	If $w$ depends solely on the $x_1$-variable, we use $\partial_1 w$ and $w'$ interchangeably. The gradient of $w$ is often split like 
	\begin{align*}
	\nabla w = (\partial_1 w| \tnabla w)\quad \text{ with}\quad \tnabla w := (\partial_2 w| \partial_3 w).
	\end{align*}
	The rescaled gradient of $w$ can then be expressed as $\nabla^\eps w = (\partial_1 w| \tfrac{1}{\eps}\tnabla w)$. 
	Note that whenever a function is defined on a subset of $\R^2$, we call its two-dimensional independent variable $\tilde x=(x_2, x_3)$, and we write $\tnabla$, $\widetilde\Delta$, and $\widetilde {\diverg}$ to indicate its gradient, Laplacian and divergence.

	If $U$ is a subset of $\R^n$, then $\overline U$ is its closure, and $\mathcal L^n(U)$, or simply $|U|$, denotes its Lebesgue measure (provided $U$ is measurable). 
	 For any open $U\subset \R^n$, we adopt the standard notation for vector-valued Sobolev spaces $H^1(U;\R^m)$ and the space of $k$-times continuously differentiable functions $C^k(\overline{U};\R^m)$. 
	The space of Lebesgue-square-integrable Banach-space-valued functions is denoted by $L^2(U;\Vcal)$ for a Banach space $\Vcal$. 
	In the case where $U$ is an interval $(a,b)\subset \R$, we shorten the notation $L^2(a,b;\Vcal):= L^2((a,b);\Vcal)$ and $H^1(a,b;\R^m):=H^1((a,b);\R^m)$. For scalar-valued functions, we often drop the image space in our notation, writing e.g.,~$H^1(U)$ instead of $H^1(U;\R)$.
	Without explicit mention, functions $(0,L)\to \R^m$ are identified with their constant extension onto $(0,L)\times U$ for $U\subset \R^2$.

	Furthermore, for any open subset $U\subset \R^2$, let
	\begin{align*}
		 H^1_{\div}(U;\R^3)= \{w\in H^1(U;\R^3) : \partial_2 w_2 + \partial_3 w_3 = 0 \text{ a.e. in $U$}\}
	\end{align*}
	and
	\begin{align*}
		L^2_0(U;\R^m) = \Bigl\{w\in L^2(U;\R^m) :   \int_U w \dd x = 0 \Bigr\}.
	\end{align*}
Moreover, we identify $L^2(0, L;H^1(U;\R^3))$ with a function in $L^2((0, L)\times U;\R^3)$, and similarly for other spaces.
	
	We employ the standard notation $\Ocal(\cdot)$ and $o(\cdot)$ for the Landau symbols.
	Finally, speaking of ``sequences'' with index $\eps>0$, means that $\eps$ can stand for any non-negative sequence $(\eps_j)_j$ with $\lim_{j\to \infty} \eps_j=0$.

	\section{Properties of the limit densities}	\label{sec:densities}
Here, we introduce and discuss relevant expressions for the formulation of the reduced limit problems, meaning the $\Gamma$-limits $\Ical^{(\alpha)}$ for $\alpha\geq 2$.

We start by defining $Q:\R^{3\times 3}\to [0,\infty)$ as the quadratic form of linearized elasticity resulting from the second derivative of the energy density $W_0$ at the identity, i.e., 
\begin{align*}
Q(F) = \nabla^2 W_0(\Id)[F, F]\quad \text{ for $F\in \R^{3\times 3}$},
\end{align*} 
cf.~(H1).
Due to (H2), Taylor expansion around the identity up to second order yields
\begin{align}\label{Taylor}
W_0(F)= \tfrac{1}{2}Q(F) +  \Ocal(|F-\Id|^3), 
\end{align}
and
along with (H3), one has that
\begin{align}\label{estimate_Q}
		Q(F) = Q(F^{\rm sym})\geq C_Q|F^{\sym}|^2,
\end{align}
for all $F\in \R^{3\times 3}$ with a constant $C_Q>0$, see e.g.~\cite{FJM02, MoM03}. 
In the following, we denote by $L$ the symmetric fourth order tensor such that
\begin{align}\label{def_Lcal}
	Q(F) = L F : F
\end{align}
	for all $F\in\R^{3\times 3}$.
	
Next, for any affine $\xi:\R^2\to \R^3$, we define $Q^\xi: H^1(\omega;\R^3)
 \to [0, \infty]$ by setting
			\begin{align}\label{Qxi}
			Q^\xi(\beta) =  \begin{cases} \displaystyle\int_\omega Q(\xi| \tnabla \beta)\dd \tilde x &\text{ if } \tr(\xi|\tnabla \beta) = 0 \text{ a.e.~in $\omega$},\\ \infty &\text{ otherwise,}\end{cases}
		\end{align}
		for $\beta\in H^1(\omega;\R^3)$. 
		Moreover, we consider the linear space subspace $\Vcal^\xi$ of $H^1(\omega;\R^3)\cap L^2_0(\omega;\R^3)$ that encompasses all functions with the property 
		\begin{align*}
	\displaystyle \int_{\omega} \widetilde \nabla \beta \dd{\tilde x} = 0  \quad \text{if $\xi(0)=0$} \qquad \text{and}\qquad
		\displaystyle \int_{\omega} \tilde x^\perp\cdot \tilde \beta \dd{\tilde x} = 0
		\quad \text{if $\xi(0)\neq 0$;} 
		\end{align*} 		
	recall the notation $\beta=(\beta_1, \tilde \beta)$. 
 	With this choice of spaces, Korn's inequality holds in the following form: There exists a constant $C_K>0$ depending only on $\omega$ such for all $\beta\in \Vcal^\xi$, 
		\begin{align}\label{Korn}
			\norm{\bigl(\tnabla \tilde \beta\bigr)^{\sym}}_{L^2(\omega;\R^{2\times 2})} \geq C_K\norm{\tnabla\tilde \beta}_{L^2(\omega;\R^{2\times 2})};
		\end{align}
		indeed, if $\xi(0)=0$, it suffices to invoke the well-known mean-value version of Korn's inequality (see e.g.~\cite{Fri47}). In the case $\xi(0)\neq 0$, on the other hand, one observes that $\Vcal^\xi$ contains no non-trivial infinitesimal rigid displacements (cf.~also \cite[Remark~4.1]{MoM04}), and hence,~\eqref{Korn} follows from~\cite[Theorem~4.4]{KoO88}. 

		 The next results provides the existence of a unique solution to the problem of minimizing the functional $Q^\xi$ from~\eqref{Qxi}.
		\begin{lemma}[Minimization of \boldmath{$Q^\xi$}]\label{lem:existence}
		Let $\xi:\R^2\to \R^3$ be affine. Then, the functional $Q^\xi$ has a unique minimizer with zero mean value, called $\beta^\xi$, which lies in $\Vcal^\xi$.
		\end{lemma}
		
		\begin{proof} 
	Let us start by observing that the constraint in $Q^\xi$ is invariant under certain affine translations; precisely, if 
		\begin{align}\label{eta}
		\eta(\tilde x) = A\tilde x +b \quad\text{ for $\tilde x\in \omega$ with $A\in \R^{3\times 2}$ such that $A_{21}+A_{32}=0$ and $b\in \R^3$,}
		\end{align}
		 then for any $\beta\in H^1(\omega;\R^3)$,
\begin{align*}
\tr(\xi|\widetilde \nabla (\beta - \eta)) =\tr(\xi|\widetilde\nabla \beta) - \tr(0|A)= \tr(\xi|\widetilde \nabla \beta).
\end{align*}  

Next, we prove that
	\begin{align}\label{minmin}
		 \inf_{\beta\in H^1(\omega;\R^3)} Q^\xi(\beta)=	\inf_{\beta\in \Vcal^\xi} Q^\xi(\beta).
	\end{align}
To this end, it suffices to show that one can find for any $\beta\in H^1(\omega;\R^3)$ with $\tr(\xi|\widetilde  \nabla \beta)=0$ a function $\eta$ as in~\eqref{eta} such that
\begin{align}\label{est1}
\beta-\eta\in \Vcal^\xi\quad \text{ and } \quad Q^\xi(\beta-\eta)\leq Q^\xi(\beta). 
\end{align}	

Indeed, for linear $\xi$, meaning $\xi(0)=0$, we specialize the coefficients in~\eqref{eta} to $b=\int_\omega \beta\dd \tilde x$ and $A=\int_{\omega} \widetilde \nabla \beta \dd \tilde x$; notice that $A_{21} +A_{32} = \int_\omega \widetilde \diverg \tilde \beta \dd{\tilde x}= -\int_\omega \xi_1\dd{\tilde x} = 0$ and $\int_\omega \beta-\eta \dd \tilde x=0$ due to~\eqref{centerofmass}. Clearly, $\int_\omega \widetilde \nabla (\beta-\eta)\dd{\tilde x}=0$, and along with $|\omega|=1$ and~\eqref{estimate_Q}, we conclude that
\begin{align*}
 Q^\xi(\beta-\eta) &= Q^\xi(\beta) +  Q(0|A) -2 \int_\omega L(\xi|\widetilde \nabla \tilde \beta) : (0|A)\dd{\tilde x} \\ &= \int_\omega Q^\xi(\beta)\dd{\tilde x} - Q(0|A)\leq  \int_\omega Q^\xi(\beta)\dd{\tilde x},
\end{align*}
recalling~\eqref{def_Lcal}. 

Otherwise, if $\xi(0)\neq 0$, take $\eta$ as in~\eqref{eta} with $A=\nu (e_3|-e_2)$ and
\begin{align*}
\nu= \frac{\int_\omega \beta\cdot\widetilde x^\perp \dd{\tilde x}}{\int_\omega |\tilde x|^2\dd{\tilde x}},
\end{align*}
as well as a translation vector $b = \int_\omega \beta \dd{\tilde x}$.
By construction, $\beta-\eta\in \Vcal^\xi$, and~\eqref{estimate_Q} in combination with the antisymmetry of $(0|A)$ implies $Q^\xi(\beta-\eta) = Q^\xi(\beta)$. 
This proves~\eqref{est1}, and thus, also~\eqref{minmin}. 

	 The existence of a minimizer of $Q^\xi$ in $\Vcal^\xi$ is a straight-forward application of the direct method, given~\eqref{estimate_Q} in combination with~\eqref{Korn} and Poincar\'e's inequality, as well as the quadratic and linear structure of $Q$ and the trace-constraint, respectively. In view of~\eqref{minmin}, then also $Q^\xi$ has a minimizer in $H^1(\omega;\R^3)$, whose uniqueness up to translations follows form the the strict convexity on symmetric matrices of the integrand of $Q^\xi$.
		\end{proof}
	
The following two remarks provide some additional insight into the properties of the minimizers $\beta^\xi$ of $Q^\xi$. First, we derive necessary conditions for the minimizers $\beta^\xi$ of $Q^\xi$ in the form of (weak) Euler-Lagrange equations; for related statements in the context of compressible rods, see~\cite[Remark 3.4]{MoM03} and \cite[Remark 4.1]{MoM04}. The second aspect concerns the linear and continuous dependence of $\beta^\xi$ on $\xi$.

	\begin{remark}[Euler-Lagrange equations]\label{rem:EL} 
	Let $\xi:\R^2\to \R^3$ be an affine function.
	
	 As a consequence of the Lagrange-multiplier theory for constrained optimization (see e.g.~\cite[Theorem 3.63]{Pey15}), 
		we obtain that $\beta$ is a minimizer of $Q^\xi$ if and only if
	
\begin{itemize}
	\item[(i)] the Euler-Lagrange equations
		\begin{align}\label{Lagrange_equation}
			\int_\omega L\big(\xi| \tnabla\beta\big):\big(0| \tnabla\phi\big)\dd \tilde x = 
			-\frac{1}{2} \int_\omega \lambda^\xi\; \widetilde{\diverg} \tilde \phi \dd \tilde x 
		\end{align}
		hold for all test functions $\phi=(\phi_1, \tilde \phi)\in H^1(\omega;\R^3)$ with a function $\lambda^\xi \in L^2(\omega)$, and 
		\item[(ii)] $\beta$ satisfies the trace condition $\tr(\xi| \widetilde \nabla \beta) = 0$, or equivalently,
			\begin{align}
			\widetilde \diverg \tilde \beta = -\xi_1.\label{trace_constraint}
			\end{align}
			\end{itemize}
		
		 Notice that the Lagrange-multiplier $\lambda^\xi$ is unique; this follows from the surjectivity of the divergence operator $\widetilde \diverg$ as a map from $H^1(\omega;\R^2)\to L^2(\omega)$, cf.~\cite[Chapter I, Corollary 2.4]{GiR86}. 

	\end{remark}
	
\begin{remark}[Linear and continuous dependence on \textbf{$\xi$}]\label{rem:H1}
The considerations in Remark~\ref{rem:EL} imply that both the Lagrange multiplier $\lambda^\xi \in L^2(\omega)$ and $\beta^\xi\in \Vcal^\xi$ (recall the definition in~Lemma~\ref{lem:existence}) depend linearly on $\xi$. Furthermore,
there exists a constant $C>0$ depending only on $Q$ and $\omega$ such that 
			\begin{align}\label{estimate_beta}
			\norm{\beta^\xi}_{H^1(\omega;\R^3)}\leq  C\norm{\xi}_{L^2(\omega;\R^3)}
		\end{align}
		holds for all affine $\xi:\R^2\to \R^3$. To see this, it is enough to exploit the minimality property of $\beta^\xi$ in conjunction with~\eqref{estimate_Q} and the inequalities of Korn and Poincar\'e.  
	Together with the last observation in Remark~\ref{rem:EL}, estimate~\eqref{estimate_beta} shows that $\xi\mapsto \beta^\xi$ is a bounded, linear map from the subspace of affine functions in $L^2(\omega;\R^3)$ into $H^1(\omega;\R^3)$, and as such also continuous. 
		\end{remark}

 We continue with a convergence statement that identifies $Q^\xi$ as the $\Gamma$-limit of a sequence of finite-valued functionals. 
	\begin{lemma}[\boldmath{$Q^\xi$} as a \boldmath{$\Gamma$}-limit]\label{lem:lower_bounds_abstract}
		For $\xi:\R^2\to \R^3$ affine and $k\in\N$, let $Q_k^\xi: H^1(\omega;\R^3) \to [0,\infty)$ by given by
		\begin{align}\label{Qxik}
			Q_k^\xi(\beta) = \int_\omega Q(\xi|\tnabla \beta)+k \tr(\xi|\tnabla \beta)^2 \dd \tilde x.
		\end{align}
		Then, $\Glim_{k\to \infty} Q_k^\xi = Q^\xi$ with respect to the weak topology in $H^1(\omega;\R^3)$.
		
		Moreover, every sequence $(\beta_k)_k\subset \Vcal^\xi$ with $\sup_{k\in \N}Q^\xi_k(\beta_k)<\infty$ admits a convergent subsequence (not relabeled) such that $\beta_k\weakly\beta$ in $H^1(\omega;\R^3)$ with  $\beta\in\Vcal^\xi$ and $\tr(\xi|\tnabla \beta)=0$.
	\end{lemma}

	\begin{proof}Fix an arbitrary affine function $\xi:\R^2\to \R^3$. 
		
		\textit{Step~1: Liminf-inequality.}  Let $\beta_k\weakly \beta$ in $H^1(\omega;\R^3)$, and assume without loss of generality that
		\begin{align*}
			\infty>\liminf_{k\to \infty} Q_k^\xi(\beta_k) = \lim_{k\to\infty} Q_k^\xi(\beta_k).
		\end{align*}
Then, $\tr(\xi|\tnabla \beta_k)\to 0$ in $L^2(\omega)$, which implies in particular that $\tr(\xi|\tnabla \beta) = 0$.
	Since $Q$ is convex, we infer by weak lower semicontinuity that
		\begin{align*}
			\liminf_{k\to\infty} Q_k^\xi(\beta_k) \geq \liminf_{k\to\infty} \int_\omega Q(\xi|\tnabla \beta_k) \dd \tilde x 
				\geq \int_\omega Q(\xi|\tnabla \beta) \dd \tilde x = Q^\xi(\beta).
		\end{align*}
			
		\textit{Step 2: Limsup-inequality.} Let $\beta\in H^1(\omega;\R^3)$ such that $\tr(\xi|\tnabla\beta) = 0$ a.e.~in $\omega$. It is immediate to see that the constant sequence $(\beta_k)_k$ with $\beta_k = \beta$ for $k\in \N$ is a recovery sequence.

\textit{Step 3: Compactness.} Let $(\beta_k)_k\subset \Vcal^\xi$ be a sequence of uniformly bounded energy for $(Q_k^\xi)_k$. 
With the help of \eqref{estimate_Q},~\eqref{Korn}, and Young's inequality, we estimate for every $k\in \N$ that
 		\begin{align*}
			  Q_k^\xi(\beta_k) &\geq \int_\omega Q(\xi|\tnabla \beta_k) \dd \tilde x \geq C_Q \int_\omega  \big|\big(\xi|\tnabla \beta_k\big)^\sym\big|^2 \dd \tilde x \\ &  \geq  C_Q \norm{(\tnabla \tilde \beta_k)^{\sym}}_{L^2(\omega;\R^{2\times 2})}^2  + \frac{C_Q}{4} \norm{\widetilde \nabla(\beta_k\cdot e_1) }_{L^2(\omega;\R^{2})}^2 - \frac{C_Q}{2} \norm{\xi}_{L^2(\omega;\R^{3})}^2\\ 
			&	\geq C\norm{\tnabla \beta_k}_{L^2(\omega;\R^{3\times 2})} - \frac{C_Q}{2} \norm{\xi}_{L^2(\omega;\R^{3})}^2
		\end{align*}
		with constants $C=C_Q\min\{C_K^2, \frac{1}{4}\}$.
		Since $\beta_k\in \Vcal^\xi$ has vanishing mean value, we conclude from Poincar\'e's inequality that $(\beta_k)_k$ is bounded in $H^1(\omega;\R^3)$. Hence, the statement follows by the weak compactness of $H^1(\omega;\R^3)$ and the weak closedness of $\Vcal^\xi$.
	\end{proof}
	
	\begin{remark}[Minimizers of \boldmath{$Q_k^\xi$}]\label{minimizers_Qxik}
			Analogous arguments to those in Lemma~\ref{lem:existence} show that for every $k\in \N$, the unique minimizer of $Q_k^\xi$ with vanishing mean value is an element of $\Vcal^\xi$. 	\end{remark}

Next, we introduce some further notation that will be needed to express the
energy densities of the intended limit problems. 
Let $Q^\ast: \R^{3\times 3}_\skw \times \R \to [0,\infty)$ be given by
	\begin{align}\label{Qast}
		\begin{split}
			\Qld(F,t)= \min\bigl\{Q^\xi(\beta) : \beta\in H^1(\omega;\R^3), \xi(\tilde x) = F(x_2e_2+x_3e_3) +te_1 \text{ for $\tilde x\in \omega$}\bigr\}.
		\end{split}
	\end{align}
	Note that $\Qld$ is well-defined according to Lemma~\ref{lem:existence}. Moreover, owing to the linear dependence of $\beta^\xi$ on $\xi$, as deduced at the end of Remark~\ref{rem:EL}, and the properties of $Q$, $\Qld$ is a positive-definite quadratic form. 
		\begin{remark}[Additive splitting of \boldmath{$Q^\ast$}]\label{rem:splitting}
		In analogy to the compressible case (see~\cite[Remark~4.4]{MoM04}), $Q^\ast$ can be split additively into two quadratic expressions that depend only on either $F$ or $t$; precisely, it holds that 
		\begin{align*}
			Q^\ast(F,t) = Q^\ast(F, 0) + \alpha t^2
		\end{align*}
		for $(F,t)\in\R^{3\times3}_{\skw}\times \R$, 
	where $\alpha\in \R$ results from a finite-dimensional constrained quadratic optimization problem, namely,
			\begin{align*}
		\alpha= \min_{a,b\in\R^3,\;a_2 + b_3 = -1} Q(e_1|a|b). 
		\end{align*}
	\end{remark}
	
	As a direct consequence of Lemma~\ref{lem:lower_bounds_abstract} (see also Remark~\ref{minimizers_Qxik}) and the classical properties of $\Gamma$-convergence, which include the convergence of minima (see e.g.~\cite{Bra02, Dal93}), we derive a useful approximation for $Q^\ast$. The next result enters into the proof of the lower bounds, cf.~Theorem~\ref{theo:rods=2}\,$ii')$ and Theorem~\ref{theo:rods>2}\,$ii)$. 
	\begin{corollary}[Pointwise approximation of \boldmath{$\Qld$}]\label{cor:convergence_Q}
	 For $k\in \N$ and $(F,t)\in \R^{3\times 3}\times \R$, let
	 \begin{align}\label{Qkxi}
	 \Qld_k(F, t)= \min\bigl\{Q^\xi_k(\beta) : \beta\in H^1(\omega;\R^3), \xi(\tilde x) = F(x_2e_2+x_3e_3) +te_1 \text{ for $\tilde x\in \omega$}\bigr\},
	 \end{align}
	 where $Q^\xi_k$ as in~\eqref{Qxik}.
	Then, $\Qld_k\to \Qld$ pointwise as $k\to \infty$. 
	\end{corollary}

We conclude this section with a brief discussion of the important special case (for applications), where $\Qld$ emerges from  an isotropic energy density $W_0$, i.e., $W_0(FS) = W_0(F)$ for all $S\in\SO(3)$ and $F\in \R^{3\times 3}$. In this situation, the minimization problem characterizing $\Qld$ can be reduced to solving a Laplace problem with suitable Neumann boundary conditions.  Under the additional geometric assumption that the cross section $\omega$ is a circle, we present a fully explicit expression for $\Qld$.
	
	\begin{example}[Isotropic case]
		If $W_0$ satisfies (H1)-(H3) and is isotropic, then the associated quadratic form is
		\begin{align}\label{Q_isotropic}
			Q(F) = \nabla^2 W_0(\Id)[F,F]= 2 \mu |F^{\sym}|^2 + \lambda (\tr F)^2, \qquad F\in \R^{3\times 3},
		\end{align}	
		with Lam\'e constants $\lambda\in\R$ and $\mu>0$, and one can show that 
		
		\begin{align}\label{Q^3_isotropic}
			\Qld(F,t)=3\mu \Bigl(F_{12}^2\int_{\omega} x_2^2 \dd \tilde x + F_{13}^2\int_\omega x_3^2\dd \tilde x +t^2\Bigr)+\mu\tau F_{23}^2
		\end{align}
		for $(F,t)\in \R^{3\times 3}_{\skw}\times \R$. 
		Here, $\tau$ denotes the torsional rigidity defined by
		\begin{align*}
			\tau := \int_\omega |\tilde x|^2 - \tilde x^\perp \cdot \widetilde \nabla \varphi \dd \tilde x,
		\end{align*}
		and $\ffi : \omega \to \R$ is a solution to the Neumann problem
		\begin{align}\label{Neumann}
			\begin{split}
				\begin{cases}
					\widetilde \Delta \varphi = 0 & \text{ in } \omega,\\
					\widetilde \nabla \varphi \cdot \nu = \tilde x^\perp\cdot \nu &\text {on } \partial \omega
				\end{cases}
			\end{split}
		\end{align}
		where $\nu$ is the outer normal vector to $\partial \omega$. 
		
		By~Corollary~\ref{cor:convergence_Q},~\eqref{Q^3_isotropic} follows from a pointwise limit procedure, once explicit expressions for $\Qld_k$ with $k\in \N$ are available. Indeed, one can extract from the literature on the theory of compressible rods, precisely, from ~\cite[Remark 3.5]{MoM03} and \cite[Remark 4.2]{Sca09}, that 
		\begin{align*}
			\Qld_k(F,t) = \frac{\mu(3\lambda + 3k + 2\mu)}{\lambda + k + \mu} \Bigl(F_{12}^2\int_{\omega} x_2^2 \dd \tilde x 
							+ F_{13}^2\int_\omega x_3^2\dd \tilde x + t^2\Bigr) + \mu\tau F_{23}^2,
		\end{align*}
		and hence, letting $k\to\infty$ implies the stated expression for $\Qld$. 

		 We point out that, in contrast to the situation without the incompressibility constraint, $Q^\ast$ in~\eqref{Q^3_isotropic} does not depend on the first Lam\'e coefficient $\lambda$. As a consistency check, observe that the trace-free constraint in~\eqref{Qxi} makes $Q^\xi$, and thus also $Q^\ast$, independent of $\lambda$. 
		\end{example}
		
		\begin{example}[Isotropic case with circular cross section]\label{ex:isotropic_circular}
		Suppose in addition to the set-up of the previous example that $\omega$ is a circle around the origin with unit measure, i.e., $\omega = \{\tilde x\in\R^2: |\tilde x|^2 \leq \frac{1}{\pi}\}$. Then, the outer unit normal vector to $\partial \omega$ becomes $\sqrt{\pi}\tilde x$, which yields a trivial solution to~\eqref{Neumann}, meaning, $\varphi=0$.
		Due to $\int_\omega x_2^2\dd \tilde x = \int_\omega x_3^2\dd \tilde x = \frac{1}{4\pi}$ and $\tau=\int_\omega x_2^2\dd \tilde x+\int_\omega x_3^2\dd \tilde x=\frac{1}{2\pi}$, formula~\eqref{Q^3_isotropic} simplifies to
		\begin{align*}
			\Qld(F,t) & = \frac{3\mu}{4\pi}\sprlr{F_{12}^2+F_{13}^2}  + \frac{\mu}{2\pi}F_{23}^2 + 3\mu t^2 \\ &= \frac{3\mu}{4\pi}|F|^2 -\frac{\mu}{4\pi} F_{23}^2 +3\mu t^2
		\end{align*}
		for $(F, t)\in \R^{3\times 3}_{\rm skew}\times \R$.
	\end{example}

\section{Technical tools for the upper bounds}\label{sec:tools}
	
	Inner perturbation arguments have proven to be useful cornerstones when it comes to the construction of locally volume-preserving deformations. The latter are needed to find recovery sequences in dimension reduction problems with an incompressibility constraint. 
	In~\cite[Proposition 5.1]{CoD06}, Conti \& Dolzmann established a first lemma of this type for $3$d-$2$d reductions in the context of incompressible membranes. Recently, the authors tailored the statement for $3$d-$1$d reductions in their work on incompressible strings~\cite{EnK19}. 
Sections~\ref{alpha=2} and~\ref{alpha>2} invoke \cite[Lemma 2.3]{EnK19} in the following formulation.
	
	\begin{lemma}[Inner perturbations]\label{lem:reparam_det=1}
		Let $\gamma, \kappa>0$ and $J\subset J'\subset \R$ be bounded closed intervals such that $0\in J$ and $J$ is compactly contained in the interior of $J'$. Further, let $Q_L:=[0, L]\times J\times J$ and $Q_L':=[0, L]\times J'\times J'$. 
		
		If $(y_\epsilon)_{\epsilon}\subset C^2(Q_L';\R^3)$ satisfies
		\begin{align*}
				\|\partial_3 y_\eps \|_{C^1(Q_L';\R^{3})}= \Ocal(\epsilon^{\kappa})
		\end{align*}
		and 
		\begin{align}\label{estimates_det}
			\|\det \rn{y} -1\|_{C^1(Q_L')}= \Ocal(\epsilon^\gamma),
		\end{align}
		then there exists 
		a sequence 
		$(u_\eps)_{\eps}\subset C^1(Q_L;\R^3)$ with 
		\begin{align*}
			\det \rn{u} = 1 \quad\text{ everywhere in $Q_L$ }
		\end{align*}	
		for $\eps$ sufficiently small, and 
			\begin{align}\label{est_varphi}
			\|u_\eps- y_\eps\|_{C^1(Q_L;\R^3)} = \Ocal(\eps^{\gamma + \kappa}). 
		\end{align}
		Replacing $\Ocal(\epsilon^\gamma)$ with $o(\epsilon^\gamma)$ in \eqref{estimates_det} yields \eqref{est_varphi} with right-hand side $o(\epsilon^{\gamma+\kappa})$ .
	\end{lemma}
	
	With the help of divergence-free extensions in the cross-section variables,
	we can prove the following approximation result, which is going to be another useful ingredient for the proof of the upper bounds in Sections~\ref{alpha=2} and~\ref{alpha>2}.
	\begin{lemma}[Approximation under divergence constraints]\label{cor:div_ext_approx}
		Let $\beta \in L^2(0,L; H^1(\omega;\R^3))$ and $\rho\in L^2((0,L)\times \omega)$ with $\rho(x_1,\cdot)$ affine for almost every $x_1\in(0,L)$ be related via
			\begin{align*}
			\widetilde \diverg \tilde \beta = \rho.
		\end{align*} 
		Further, let $(\rho_\delta)_{\delta}\subset C^2([0,L]\times \R^2)$ be a sequence of functions that are affine in the cross-section variables satisfying $\rho_\delta \to \rho$ in $L^2((0, L)\times \omega)$ as $\delta\to 0$.

		Then, there exists a sequence $(\beta_\delta)_{\delta}\subset C^2([0,L]\times \R^2;\R^3)$ with
		\begin{align*} 
			\widetilde \diverg \tilde \beta _\delta= \rho_\delta
		\end{align*}
		for every $\delta$ and $\beta_\delta \to \beta$ in $L^2(0,L;H^1(\omega;\R^3))$ as $\delta\to 0$.
	\end{lemma}

	\begin{proof}
		Due to the structural properties of $\rho$ and $\rho_\delta$, one can find $a,b,c\in L^2(0,L)$ and $a_\delta,b_\delta,c_\delta\in C^2([0,L])$ such that 
		\begin{align*}
			\rho(x) &= a(x_1)x_2 + b(x_1)x_3 + c(x_1), \\
			\rho_\delta(x) &= a_\delta(x_1)x_2 + b_\delta(x_1)x_3 + c_\delta(x_1).
		\end{align*}
		With the definitions
		\begin{align*}
			\Xi(x) &:= \tfrac{1}{2}\bigl(a(x_1)x_2^2 + c(x_1)x_2\bigr) e_2 + \tfrac{1}{2}\sprlr{b(x_1)x_3^2 + c(x_1)x_3}e_3,\\
			\Xi_\delta(x) &:=\tfrac{1}{2}\bigl(a_\delta(x_1)x_2^2 + c_\delta(x_1)x_2\bigr)e_2 + \tfrac{1}{2}\sprlr{b_\delta(x_1)x_3^2 + c_\delta(x_1)x_3}e_3,
		\end{align*}
		it holds that 
		\begin{align}\label{359}
		\Xi_\delta\to \Xi\quad  \text{in $L^2((0, L)\times \omega)$ as $\delta\to 0$,}
		\end{align}
		and by straight-forward calculation,
		\begin{align}\label{357}
			\partial_2 \Xi_2 + \partial_3 \Xi_3   = \rho \qquad \text{and}\qquad 
			\partial_2(\Xi_\delta\cdot e_2) + \partial_3(\Xi_\delta\cdot e_3)   = \rho_\delta.
		\end{align} 
		 	Hence, $\beta - \Xi\in L^2(0, L;H^1_{\div}(\omega;\R^3))$, and after divergence-free extension in the cross-section variables according to~\cite[Proposition 3.1, Corollary 3.2]{KMPT00}, one can view $\beta -\Xi$ as an element of $L^2(0, L;H^1_{\div}(\R^2;\R^3))$. 
		
		A standard mollification argument yields a sequence $(\hat\beta_\delta)_{\delta}\subset C^2([0,L]\times \R^2;\R^3)$ of functions that are divergence-free in the last two variables, i.e., for any $\delta$
		\begin{align}\label{358}
			\partial_2 (\hat\beta_\delta\cdot e_2)+ \partial_3 (\hat\beta_\delta\cdot e_3) = 0 \quad \text{in $[0,L]\times \R^2$,}
		\end{align}
  such that $\hat\beta_\delta\to \beta - \Xi$ in $L^2(0, L;H^1(\R^2;\R^3))$ as $\delta\to 0$. 
  	
		Finally, in view of~\eqref{357} and~\eqref{358} as well as~\eqref{359}, setting $\beta_\delta = \hat\beta_\delta + \Xi_\delta$ provides the desired sequence. 
	\end{proof}

\section{The regime $\alpha = 2$}\label{alpha=2}

The following $\Gamma$-convergence theorem is the first main result of this paper. It provides a reduced one-dimensional model for incompressible rods, which involves, besides the deformation of the mid-fiber, quantities related to bending and torsion effects. The limit procedure turns the local volume preservation in the three-dimensional model from non-linear elasticity into a trace constraint.

	\begin{theorem}[$\Gamma$-limit for \boldmath{$\alpha=2$}]\label{theo:rods=2} 
	Let $\Ical_\eps^{(2)}$ for $\eps>0$ be the functional introduced in~\eqref{Ical_eps^alpha} with $\alpha=2$. Moreover, let
		\begin{align}\label{I^2}
		\begin{split}
			\Ical\ui{2} : H^2(0,L;\R^3)&\times H^1(0,L;\R^{3\times 2}) \to [0,\infty],\\
			&(u,D) \mapsto \begin{cases}
									\displaystyle\frac{1}{2} \int_0^L \Qld(A(x_1), 0) \dd x_1&\text{ for } (u,D) \in \Acal\ui{2},\\
									\infty &\text{ otherwise, }
								 \end{cases}
		\end{split}
	\end{align}
	where $A:=R^TR'$ with $R:=(u'|D)$, $\Qld$ is defined in~\eqref{Qast}, and 
	\begin{align*}
			\Acal\ui{2} := \{(u,D)\in  H^2(0,L;\R^3)\times H^1(0,L;\R^{3\times 2}): (u'|D)\in\SO(3) \text{ a.e.~in $(0, L)$}\}.
	\end{align*}
	
		\textit{i) (Compactness)} For every sequence $(u_\eps)_\eps\subset H^1(\Omega;\R^3)\cap L^2_0(\Omega;\R^3)$ and  $\sup_{\eps>0} \Ical_\eps\ui{2}(u_\eps) <\infty$, 
		there exists a subsequence (not relabeled) and $(u,D)\in\Acal\ui{2}$ such that
		\begin{align}\label{conv_2}
			\begin{split}
				u_\eps &\to u \text{ in }H^1(\Omega;\R^3),\\
				\frac{1}{\eps}\tnabla u_\eps &\to D\text{ in }L^2(\Omega;\R^{3\times 2}).
			\end{split}
		\end{align}
		
		\textit{ii) (Variational limit)} The sequence $(\Ical_\eps\ui{2})_\eps$ $\Gamma$-converges for $\eps\to 0$ to $\Ical\ui{2}$ with respect to the convergence \eqref{conv_2}, that is, the following two conditions are fulfilled:
		\begin{itemize}
			\item[$ii')$] (Lower bound) Let $(u_\epsilon)_{\epsilon}\subset H^1(\Omega;\R^3)$ satisfy $\eqref{conv_2}$ for $(u,D)\in\Acal\ui{2}$, then 
			\begin{align*}
				\liminf_{\epsilon\to 0} \Ical_\epsilon\ui{2}(u_\epsilon) \geq \Ical\ui{2}(u,D);
			\end{align*}
			
			\item[$ii'')$] (Upper bound) For every $(u,D)\in\Acal\ui{2}$ there exists a sequence $(u_\epsilon)_{\epsilon}\subset H^1(\Omega;\R^3)$ satisfying \eqref{conv_2} and
			\begin{align*}
				\limsup_{\epsilon\to 0} \Ical_\epsilon\ui{2}(u_\epsilon) \leq \Ical\ui{2}(u,D).
			\end{align*}
		\end{itemize}
		
	\end{theorem}

	\begin{proof} Ad $i)$. 
		Let $(u_\eps)_\eps$ be a sequence with uniformly bounded energy and vanishing mean value. 
		It follows from $W_0\leq W$ and hypothesis (H2) that
		\begin{align*}
			\frac{C_1}{\eps^2}\dist^2(\rn{u},\SO(3))\leq\frac{1}{\eps^2}\int_\Omega W_0(\rn{u}) \dd x \leq \Ical_\eps\ui{2}(u_\eps) \leq C
		\end{align*}
		for a constant $C>0$. The statement $i)$ is now an immediate consequence of the compactness result in~\cite[Theorem 2.1]{MoM03}.

		Ad $ii')$. Recalling the definitions of the energy densities $W_k$ in \eqref{penalty_density} and the associated auxiliary functionals $\Ical_{k,\eps}\ui{2}$ from \eqref{penalty_energy},
		 we obtain
		\begin{align*}
			 \Ical\ui{2}_\epsilon(u_\eps)\geq \Ical\ui{2}_{k,\epsilon}(u_\eps)
		\end{align*}
 for every $\eps>0$ and $k\in\N$.
		The lower bound in the compressible case \cite[Theorem 3.1]{MoM03} yields that
		\begin{align}\label{ineq2}
			 \liminf_{\eps\to 0} \Ical\ui{2}_\epsilon(u_\eps)\geq \liminf_{\eps\to 0} \Ical\ui{2}_{k,\epsilon}(u_\eps) \geq \Ical_k\ui{2}(u,D),
		\end{align}
	where
		\begin{align*}
			\Ical_k\ui{2}(u,D)= \frac{1}{2}\int_0^L \Qld_k(A(x_1), 0) \dd x_1
		\end{align*}
		with $A=R^TR'$, $R=(u'|D)$, and $\Qld_k$ defined as in \eqref{Qkxi}. 
		In view of Corollary~\ref{cor:convergence_Q} and the monotonicity of $\Qld_k$ with respect to $k$, meaning, $\Qld_k\leq \Qld_{k+1}$ for all $k\in \N$, the theorem on monotone convergence implies that $\Ical_k\ui{2}(u,D) \to \Ical\ui{2}(u,D)$ for $k\to \infty$. Thus, together with~\eqref{ineq2}, we finally conclude that 
		\begin{align*}
			\liminf_{\eps\to 0} \Ical\ui{2}_\epsilon(u_\epsilon) \geq \lim_{k\to \infty} \Ical\ui{2}_k(u,D) = \Ical\ui{2}(u,D).
		\end{align*}
		
	Ad $ii'')$. Let us fix $(u,D)\in\mathcal A\ui{2}$. We split the proof of the upper bound into three steps, starting with the construction of recovery sequences in the case when we have extra regularity for the refined limit deformations.

		\textit{Step 1: Recovering smooth limit functions.}
		Let $u\in C^3([0,L];\R^3)$, $D\in C^2([0,L];\R^{3\times 2})$ such that $R=(u'|D)\in\SO(3)$ everywhere in $[0,L]$.
		Moreover, suppose that $\beta \in C^2([0,L]\times \R^2;\R^3)$ satisfies 
		\begin{align}\label{trace_vanishing_=2}
			\tr \big(A(x_1)(x_2 e_2 + x_3e_3)|\tnabla \beta(x)\big) = 0
		\end{align}
		 for every $x\in [0,L]\times \R^2$. Note that in the following, we drop the arguments $x_1$ and $x$ in our notation when they are clear from the context. 
		 
		Now, let $J\subset J'\subset \R$ be two intervals as in Lemma \ref{lem:reparam_det=1} with $\omega\subset J\times J$.
		Inspired by the recovery sequence in the situation without incompressibility, see~\cite[Theorem 3.1]{MoM03}, we define for every $\eps>0$ and $x\in Q_L':=[0,L]\times J'\times J'$, 
		\begin{align*}
			y_\epsilon(x) = u(x_1) + \eps R(x_1)(x_2e_2+x_3e_3) + \epsilon^2 R(x_1) \beta(x). 
		\end{align*}
		By construction, $(y_\eps)_\eps\subset C^2(Q_L';\R^{3})$ converges to $(u, D)$ in the sense of \eqref{conv_2} and has the property that 
		\begin{align}\label{con1}
		\norm{\partial_3 y_\eps}_{C^1(Q_L';\R^{3\times 3})} = \Ocal(\eps).
		\end{align} 
		The rescaled gradient of $y_\eps$ is given by
		\begin{align}\label{formula_rescaled}
			\rn{y} = R +  \epsilon R\big(A(x_2e_2+x_3e_3) |\tnabla\beta\big) + \epsilon^2\partial_1(R\beta)\otimes e_1,
		\end{align}
	and hence, 
		\begin{align*}
			\det&(\rn{y}) = \det\left(R^T\rn{y}\right) 
			 = 1 + \epsilon\tr\big(A(x_2 e_2 + x_3 e_3)|\tnabla \beta\big) + \Ocal(\epsilon^2),
		\end{align*}
		in view of the identity $\det(\Id + F) = 1 + \tr F + \tr \cof F + \det F$ for every $F\in\R^{3\times 3}$. 
		Together with the vanishing trace assumption \eqref{trace_vanishing_=2}, we conclude that
		\begin{align}\label{con2}
			\norm{\det (\rn{y}) - 1}_{C^1(Q_L')} = \Ocal(\eps^2).
		\end{align}

		In light of~\eqref{con1} and~\eqref{con2}, we can now apply Proposition~\ref{lem:reparam_det=1} with the choices $\gamma=2$ and $\kappa=1$ to obtain a sequence $(u_\epsilon)_{\epsilon}\subset C^1(\overline{\Omega};\R^3)$ 
		satisfying 
		$\det \rn{u} = 1$ everywhere in $\overline{\Omega}$ 
		and
		\begin{align}\label{Oeps3}
			\norm{y_\eps - u_\eps}_{C^1(\overline{\Omega};\R^3)} = \Ocal(\eps^3);
		\end{align}	
		due to the convergence behavior of $(y_\eps)_\eps$, this shows in particular that $(u_\eps)_\eps$ converges to $(u,D)$ in the sense of \eqref{conv_2} as well. 	
		
		Moreover, the combination of~\eqref{formula_rescaled} and~\eqref{Oeps3} gives that
		\begin{align*}
			R^T\rn{u} = \Id  + \eps \big(A(x_2 e_2+x_3e_3)|\tnabla \beta\big) + \Ocal(\epsilon^2), 
		\end{align*} 
		and hence, by the Taylor expansion in~\eqref{Taylor},
		\begin{align*} 
			W(\rn{u}) = W_0(\rn{u}) &= W_0(R^T\rn{u}) = \frac{\epsilon^2}{2}Q\big(A(x_2 e_2+x_3 e_3)|\tnabla \beta\big) + \Ocal(\epsilon^3);
		\end{align*}
		also, we have used here~\eqref{def_W} under consideration of $\det \rn{u} = 1$ and the assumption of frame indifference (H3). 
	Altogether, this shows that
		\begin{align*}
	\limsup_{\epsilon\to 0} \Ical\ui{2}_\eps(u_\eps)
	\leq \frac{1}{2}\int_\Omega  Q\big(A(x_2 e_2+x_3 e_3)|\tnabla \beta\big)\dd x.
		\end{align*}
		
		\textit{Step 2: Approximation and optimization.}
		To approximate $(u,D)\in\mathcal A\ui{2}$ suitably by smooth functions, we invoke the same argument as in~\cite[Theorem 3.1]{MoM03}, which provides a sequence 
		$(R_\delta)_\delta\subset  C^2([0,L];\SO(3))$ such that
		 		\begin{align*} 
		R_\delta \to R=(u'|D) \quad \text{in $H^1(0, L;\R^{3\times 3})$ for $\delta\to 0$,}
		\end{align*} 
and thus, $(R_\delta)_\delta$ converges to $R$ also uniformly, which implies in particular that
		\begin{align}\label{convergence_R_delta}
			A_\delta:=(R_\delta)^TR_\delta'\to R^TR' =A\text{ in }L^2(0,L;\R^{3\times 3}).
		\end{align}  
		For any $\delta$, we consider 
		\begin{align*}
			u_\delta(x_1) :=\int_0^{x_1}R_\delta(t)e_1\dd t - c_\delta \quad \text{and} \quad D_\delta(x_1) := (R_\delta(x_1)e_2|R_\delta(x_1)e_3) \quad \text{ for $x_1\in(0,L)$,}
		\end{align*}
		 where $c_\delta\in\R$ is chosen in such a way that $u_\delta$ has the same mean value as $u$. Then, $R_\delta = (u_\delta'| D_\delta)$, and 
		 \begin{align}\label{convergence_u,D_delta}
			\begin{split}
				u_\delta \to u &\text{ in } H^2(0,L;\R^3),\\
				D_\delta \to D &\text{ in } H^1(0,L;\R^{3\times 2})
			\end{split}
		\end{align}
		as $\delta\to 0$.

Next, we introduce the functions 
               \begin{align*}
			 \xi(x) & := A(x_1)(x_2e_2 + x_3e_3), \\ 
			  \xi_\delta(x) & := A_\delta(x_1)(x_2e_2 + x_3e_3),
		\end{align*}
		for $x\in (0, L)\times \omega$, and take 
	 $\beta_A(x_1, \cdot)$ for $x_1\in (0, L)$ as the unique solution to the minimization problem defining $Q^\ast(A(x_1), 0)$, that is, 
	 \begin{align}\label{beta_applied}
	 \beta_A(x) =\beta^{\xi(x_1, \cdot)}(\tilde x) \quad\text{for $x=(x_1, \tilde x)\in (0,L)\times \omega$,}
	 \end{align}
	 cf.~\eqref{Qast} and Lemma~\ref{lem:existence}. Notice that in light of Remarks~\ref{rem:H1} and~\ref{rem:EL}, $\beta_A\in L^2(\omega;H^1(\omega;\R^3))$. 
		
	Considering~\eqref{convergence_R_delta}, Corollary~\ref{cor:div_ext_approx} applied with $\rho= - \xi\cdot e_1$, $\rho_\delta= - \xi_\delta\cdot e_1$ and $\beta$ as in~\eqref{beta_applied} gives rise to 
	a sequence $(\beta_\delta)_{\delta}\in C^2([0,L]\times \R^2;\R^3)$ that satisfies the trace condition 
		\begin{align*}
			 \tr (\xi_\delta|\tnabla \beta_\delta) = 0
		\end{align*}
		on all of $[0,L]\times \R^2$ and
		\begin{align}\label{convergence_beta_delta}
			\beta_\delta\to \beta \text{ in } L^2(0,L;H^1(\omega;\R^3)). 
		\end{align}	
		
		\textit{Step 3: Diagonalization.}
		For every $\delta$, we repeat Step~1 with $u = u_\delta$, $D = D_\delta$ and $\beta = \beta_\delta$ to obtain a sequence
		$(u_{\delta,\epsilon})_{\epsilon}\subset C^1(\overline \Omega;\R^3)$ satisfying $\det \nabla^\epsilon u_{\delta,\epsilon} = 1$ for all $\epsilon>0$ sufficiently small, and 
		\begin{align}\label{approx}
			\begin{split}
				u_{\delta,\epsilon} \to u_\delta &\text{ in } H^1(\Omega;\R^3),\\
				\frac{1}{\eps}\tnabla u_{\delta,\epsilon} \to D_\delta &\text{ in } L^2(\Omega;\R^{3\times 2}),
			\end{split}
		\end{align}			
		as well as
		\begin{align}\label{diag}
			\limsup_{\delta\to 0}\limsup_{\epsilon\to 0} \Ical_{\epsilon}\ui{2}(u_{\delta,\epsilon}) \leq \limsup_{\delta\to0}\frac{1}{2} \int_\Omega   Q\big(A_\delta(x_2e_2+x_3e_3) |\tnabla \beta_\delta\big) \dd x  = \Ical\ui{2}(u,D).
		\end{align}
		For the last equality, we have exploited \eqref{convergence_R_delta}, \eqref{convergence_beta_delta}, the optimality of $\beta$ from~\eqref{beta_applied}, and the fact that $Q$ is a quadratic form. 
		
		Finally, we extract a diagonal sequence $(u_\eps)_\eps$ from $(u_{\delta,\eps})_{\delta,\eps}$ in the sense of Attouch \cite[Lemma 1.15,1.16]{Att84} to conclude the proof of the upper bound. Indeed, combining \eqref{convergence_u,D_delta} with \eqref{approx}, and \eqref{diag} gives $(u_\eps)_\eps$ satisfying \eqref{conv_2} and
		\begin{align*}
			\limsup_{\eps\to 0}\Ical\ui{2}_\eps(u_\eps) \leq \Ical\ui{2}(u,D),
		\end{align*}
		as claimed.
	\end{proof}
	
	\begin{remark}[Incorporating external forces]  
	The statement of Theorem~\ref{theo:rods=2} still holds if we replace the sequence of elastic energy functionals $(\Ical_\eps\ui{2})_\eps$ with the system energies $(\Jcal\ui{2}_\eps)_\eps$ as in \eqref{Jcal_eps^alpha}; the corresponding $\Gamma$-limit (with respect to the convergence \eqref{conv_2}) is given by
		\begin{align}\label{Jcal22}
			\Jcal\ui{2}(u,D)= \Ical\ui{2}(u,D) - \int_0^L f\cdot u \dd x_1
		\end{align}	
		for $(u, D)\in H^2(0,L;\R^3)\times H^1(0,L;\R^{3\times 2})$, given that the external force term constitutes  a continuous perturbation. 		
		
		Furthermore, we observe that introducing 
		\begin{align}\label{def_F}
			h:(0,L)\to \R^3,\quad t\mapsto \int_0^t f(x_1) \dd x_1,
		\end{align}
		as the primitive of $f$ allows us, in view of~\eqref{hyp1_f}, to rewrite~\eqref{Jcal22} as
		\begin{align}\label{J^2}
			\Jcal\ui{2}(u,D) = \Ical\ui{2}(u,D) + \int_0^L h\cdot u' \dd x_1
		\end{align} 
		for $(u,D)\in\Acal\ui{2}$. Hence, just like $\Ical\ui{2}$, the functional $\Jcal\ui{2}$ is invariant under translation.
	\end{remark}

In the second part of this section, we complement the asymptotic analysis of the sequence $(\Jcal\ui{2}_\eps)_\eps$ by calculating the Euler-Lagrange equations of the limit functional $\Jcal\ui{2}$, which characterize its stationary points. 
	First, let us briefly introduce the necessary notation of the stress and its moments of first order.
	For $(u,D)\in\Acal\ui{2}$, let
	\begin{align}\label{defA1}
		A := R^TR'  \in L^2(\omega;\R^{3\times 3}_{\rm skew}), 
	\end{align} 
	recalling that $R=(u'|D)$. 
	Furthermore, let $\beta_A\in L^2(0,L;H^1(\omega;\R^3))$ be such that $\beta_A(x_1,\cdot)$ is a solution to the variational problem defining $\Qld(A(x_1), 0)$ for $x_1\in (0, L)$, cf.~\eqref{Qast}  and~\eqref{beta_applied}.
	 The stress $M\in L^2((0, L)\times \omega;\R^{3\times 3})$ associated with $(u,D)$ is then given as
	\begin{align}\label{stress}
		M = L\big(A(x_2e_2+x_3e_3)| \tnabla\beta_A\big)
	\end{align}
	with $L$ as in \eqref{def_Lcal}, and $\hat M, \check M\in L^2(0,L;\R^{3\times 3})$ denote the first-order moments of $M$, i.e.,
	\begin{align}\label{first_moments}
		\hat M = \int_\omega x_2 M \dd \tilde x\quad \text{and} \quad \check M = \int_\omega x_3 M \dd \tilde x.
	\end{align}
	
	\begin{proposition}[Euler-Lagrange equations]\label{prop:EL-eq2}			
		Let $\Jcal\ui{2}$ be as in~\eqref{J^2} with~\eqref{def_F} and~\eqref{hyp1_f}.  
		Then, $(u,D)\in \Acal\ui{2}$ is a stationary point of $\Jcal\ui{2}$ if and only if 
		\begin{align*}
				\hat M_{11}'-\hat M_{22}' &= A_{13}(\check M_{21} - \hat M_{31}) - A_{23}(\check M_{11}- \check M_{33}) - h\cdot De_1 ,\\
			 \check M_{11}' -\check M_{33}' &=- A_{12}(\check M_{21} -\hat M_{31}) +  A_{23}(\hat M_{11} - \hat M_{22})  - h\cdot De_2,\\
			\check M_{21}' - \hat M_{31}' &= A_{12}(\check M_{11} - \check M_{33})	-A_{13}(\hat M_{11}-\hat M_{22}) ,
			\end{align*}
			and
			\begin{align*}
				\hat M_{11}(0) - \hat M_{22}(0) & = \hat M_{11}(L) - \hat M_{22}(L) = 0,\\
				\check M_{11}(0) - \check M_{33}(0) &= \check M_{11}(L) - \check M_{33}(L) = 0,\\
				\check M_{21}(0) - \hat M_{31}(0) &= \check M_{21}(L) - \hat M_{31}(L) = 0,	
		\end{align*}
		where $A,\hat M,\check M$ are defined as in~\eqref{defA1} and~\eqref{first_moments}, respectively.
	\end{proposition}
	
	\begin{proof}
 The calculation of the first variation of $\Jcal\ui{2}$, which we will identify with a functional on $H^1(0, L;\SO(3))$ in the following, can be done similarly to~\cite[Lemma 2.3]{MoM08}. Precisely, for any $(u, D)\in \Acal\ui{2}$
 and $B\in H^1(0, L;\R^{3\times 3}_{\rm skew})$, we consider a curve 
		\begin{align*}
		\gamma : (-1,1) \to H^1(0, L;\SO(3))\qquad\text{ with $\gamma(0) = R=(u'|D)$ and $\partial_s\gamma(0) = RB$;}
		\end{align*} 
		notice that the tangent space of $H^1(0, L;\SO(3))$ at $R$ can be identified with $R H^1(0, L;\R^{3\times 3}_{\rm skew})$. 
		Evaluating $\Jcal\ui{2}$ along this curve gives 
		\begin{align*}
			\Jcal\ui{2}(\gamma(s)) &= \frac{1}{2}\int_0^L \Qld(\gamma(s)^T\gamma(s)', 0) \dd x_1 + \int_0^L h\cdot \gamma(s)e_1 \dd x_1
		\end{align*} for $s\in(-1,1)$. In view of
		\begin{align}\label{derivative_gamma}
			\frac{\dd}{\dd s}\restrict{s=0} \gamma(s)^T\gamma(s)' &= - B (R^TR') + (R^TR')B + B' \nonumber \\ &= AB-BA+B'=: H \in L^2(0, L;\R^{3\times 3}_{\rm skew}),
		\end{align} 
		we find that
		\begin{align}\label{split}
			\begin{split}
				\frac{\dd}{\dd s}\restrict{s=0} \Jcal\ui{2}(\gamma(t)) =
				\int_\Omega  &M : \sprlr{H(x_2e_2+x_3e_3)\big| 0 \big | 0}\dd x\\
					&+	\int_\Omega M : (0|\tnabla \beta_H) \dd x + \int_0^L  h\cdot RBe_1 \dd x_1,
			\end{split}
		\end{align}
where $\beta_H\in L^2(0,L; H^1(\omega;\R^3))$ is such that $\beta_H(x_1, \cdot)$ solves the minimization problem in~\eqref{Qast} with the argument $(H(x_1), 0)$, cf.~Remark~\ref{rem:H1}.

		In view of~\eqref{first_moments}, the first integral in~\eqref{split} can be rewritten as
		\begin{align}\label{first_summand=2}
			\int_0^L  \hat Me_1\cdot He_2 + \check M e_1 \cdot He_3\dd x_1.
		\end{align} 
		The treatment of the second term in \eqref{split} exploits the Euler-Lagrange equations~\eqref{Lagrange_equation} and the trace condition~\eqref{trace_constraint} applied to every affine function $\xi(\tilde x)=H(x_1)(x_2e_2 + x_3e_3)$ with $x_1\in (0, L)$. If we write $\lambda_H(x_1, \cdot)$ for the corresponding Lagrange multipliers, $\lambda_H$ can be viewed as an element in $L^2((0, L)\times \omega)$ due to the linear dependence on the affine input pointed out in Remark~\ref{rem:EL}.
		Let us introduce $\hat\lambda_H,\check\lambda_H \in L^2(0, L)$ as the first moments of $\lambda_H$,  that is,
		\begin{align*}
			\hat\lambda_H(x_1) = \int_\omega \lambda_H x_2 \dd \tilde x\quad \text{and} \quad  \check\lambda_H(x_1) = \int_\omega \lambda_H x_3 \dd \tilde x.
		\end{align*}
		Then,
		\begin{align}\label{second_summand=2}
			\int_\Omega M : (0|\tnabla \beta_H) \dd x &= -\frac{1}{2} \int_\Omega \lambda_H\, \widetilde \diverg \tilde \beta_H\dd x=\frac{1}{2}\int_\Omega \lambda_H (H_{12}x_2+H_{13}x_3) \dd x \nonumber \\ 
				&=\frac{1}{2}\int_0^L \hat\lambda_H e_1\cdot He_2+ \check\lambda_H e_1 \cdot He_3 \dd x_1.
		\end{align}
		
		On the other hand, the choice of test fields $\phi = (0,\frac{1}{2}x_2^2,0)$ and $\phi = (0,0,\frac{1}{2}x_3^2)$ in \eqref{Lagrange_equation} yields that
		\begin{align}\label{Lagrange_lambdaH}
			\hat\lambda_H = -2\hat M_{22}\quad\text{ and }\quad \check\lambda_H = -2\check M_{33}.
		\end{align}
		Therefore, by joining~\eqref{split} with \eqref{first_summand=2} and \eqref{second_summand=2}, we find that
		\begin{align}\label{total}
			\frac{\dd}{\dd t}\restrict{t=0} \Jcal\ui{2}(\gamma(t))= \int_0^L  &(\hat Me_1 -  \hat M_{22} e_1)\cdot He_2 \nonumber\\
			&\qquad + (\check M e_1 - \check M_{33} e_1)\cdot He_3 +  h\cdot RBe_1 \dd x_1.
		\end{align}
		
To conclude, it suffices now to specialize $B$ to three classes of test fields, recalling that $H$ depends on $B$ through~\eqref{derivative_gamma}: For $\psi, \theta, \sigma \in H^1(0,L)$, we plug
		\begin{align*}
			\begin{split}
				B &= \psi e_1\otimes e_2 - \psi e_2\otimes e_1,\\
				B &= \theta e_1\otimes e_3 - \theta e_3\otimes e_1,\\
				B &= \sigma e_2\otimes e_3 - \sigma e_3\otimes e_2,
			\end{split}
		\end{align*}
		into~\eqref{total},
which gives rise to the system 
		\begin{align*}
				\displaystyle\int_0^L (\hat M_{11}-\hat M_{22})\psi' + A_{13}(\check M_{21} - \hat M_{31})\psi - A_{23}(\check M_{11} - \check M_{33})\psi  - h \cdot Re_2 \psi \dd x_1 &= 0,\\
				\displaystyle\int_0^L  (\check M_{11} - \check M_{33})\theta'  - A_{12}(\check M_{21} - \hat M_{31})\theta + A_{23}(\hat M_{11}-\hat M_{22})\theta- h \cdot Re_3 \theta \dd x_1 &= 0,\\
				\displaystyle\int_0^L (\check M_{21} - \hat M_{31})\sigma'  + A_{12}(\check M_{11} - \check M_{33})- \sigma   A_{13}(\hat M_{11}-\hat M_{22})\sigma \dd x_1 &= 0.
			\end{align*}
This corresponds to the weak formulation of the stated equations and boundary conditions. 
	\end{proof}

	In the special case of isotropic incompressible rods with circular cross section, the characterizing equations for stationary points simplify considerably.

	\begin{example}  
	 We adopt the setting of Example~\ref{ex:isotropic_circular}, that is, $\omega$ is a circle around the origin with unit measure and $W_0$ is supposed to be isotropic, which implies that $Q$ is of the form \eqref{Q_isotropic} with Lam\'e coefficients $\lambda\in \R$ and $\mu>0$.
		Then, under the assumptions of Proposition \ref{prop:EL-eq2}, $(u,D)\in\Acal\ui{2}$ is a stationary point of $\Jcal\ui{2}$ if and only if
		\begin{align}\label{system_isotropic}
			\begin{cases}
				A_{12}' = - \frac{4\pi}{3\mu} h\cdot De_1,\quad &A_{12}(0) = A_{12}(L) = 0,\\		
				A_{13}' = -\frac{4\pi}{3\mu} h\cdot De_2,\quad &A_{13}(0) = A_{13}(L) = 0,\\
				A_{23}=0.
			\end{cases}
		\end{align} 

		Indeed, in this situation, $L$ takes the form
		\begin{align*}
			L F = 2\mu F^{\sym} + \lambda \tr(F)\Id \quad \text{for $F\in\R^{3\times3},$}
		\end{align*} 
		with $\mu>0$ and $\lambda\in \R$, and $\beta_A$ can be determined to be 
		\begin{align*}
			\beta_A(x) = -\tfrac{1}{4}(A_{12}(x_2^2-x_3^2) + 2A_{13}x_2x_3))e_2 -\tfrac{1}{4}(A_{13}(x_3^2-x_2^2) + 2A_{12}x_2x_3))e_3 
		\end{align*}		
		for $x\in (0, L)\times \omega$;
		note that $\beta_A$ emerges from the corresponding expression in the compressible case as the limit for diverging first Lam\'e coefficient.
		Then, the stress as defined in~\eqref{stress} becomes
		\begin{align*}
			M &= 2\mu \big(A(x_2e_2+x_3e_3)| \tnabla\beta_A\big)^{\rm sym} \\& = \mu\begin{pmatrix}
					2(A_{12}x_2+A_{13}x_3) & A_{23} x_3 & -A_{23}x_2\\A_{23}x_3 & - A_{12}x_2-A_{13}x_3  & 0 \\ -A_{23}x_2 & 0 & - A_{12}x_2-A_{13}x_3 
				\end{pmatrix}
		\end{align*}
		and, in view of \eqref{centerofmass}, the first bending moments are
		
		\begin{align*}
			\hat M = \frac{\mu}{4\pi}\begin{pmatrix}
					2A_{12} & 0 & -A_{23}\\0& - A_{12} & 0 \\ -A_{23} & 0 & - A_{12}
				\end{pmatrix}\quad\text{and}\quad  \check M = \frac{\mu}{4\pi}\begin{pmatrix}
					2A_{13} & A_{23} &0\\A_{23}& -A_{13} & 0 \\ 0 & 0 & -A_{13}
				\end{pmatrix}. 
		\end{align*}
		Finally, we insert these expressions into the equations of Proposition~\ref{prop:EL-eq2},
		 which gives rise to~\eqref{system_isotropic}.
	\end{example}
	
	We conclude the study of the regime $\alpha=2$ with a brief comparison of the Euler-Lagrange equations for rods with and without a local volume-preservation constraint.

	\begin{remark}[Comparison with compressible rods]
 a) The difference between the result of Proposition~\ref{prop:EL-eq2} and \cite[Lemma~2.3]{MoM08} lies in the presence of non-trivial bending terms $\hat M_{22}$ and $\check M_{33}$. The latter arise as moments of the Lagrange multipliers that are necessary to accommodate the trace constraint in the minimization problem defining $\Qld$, cf.~\eqref{Lagrange_lambdaH} and Remark~\ref{rem:EL}.

 b) In the special case of rods with circular cross-section of isotropic material,
the  structure of the Euler-Lagrange equations is identical, but the constant coefficients vary. To be more precise, the analogue of the factor $\frac{4\pi}{3\mu}$ in~\eqref{system_isotropic} is $\frac{4\pi(\lambda+\mu)}{\mu(3\lambda+2\mu)}$ when the assumption of incompressibility is dropped. The connection between these factors becomes apparent in the limit of diverging first Lam\'e coefficients. 
 \end{remark}

\section{The regimes $\alpha>2$}\label{alpha>2}
This section covers the asymptotic analysis in all the remaining scaling regimes. 
Like in the setting without incompressibility, these regimes share the common feature that the limit deformations correspond to rigid body motions. 	
In order to extract more refined information on the reduced limit problems, it is useful to estimate the deviation of low energy sequences $(u_\eps)_\eps\subset H^1(\Omega;\R^3)$ (after suitable translation, global rotation, and scaling) from the identity. 
To this end, we follow~\cite{MoM04, Sca09} in considering sequences $(v_\eps)_\eps\subset H^1(0,L;\R^2), (w_\eps)_\eps\subset H^1(0,L)$ given by
	\begin{align}\label{def_v,w}
		\begin{split}			
			v_\epsilon &= \frac{1}{\epsilon^{\alpha-2}}\int_\omega \tilde u_\eps \dd \tilde x,\\
			w_\epsilon &=  \frac{1}{\epsilon^{\alpha-1}} \left(\textstyle\int_\omega |\tilde x|^2\dd \tilde x\right)^{-1}\int_\omega \tilde u_\eps\cdot \tilde x^\perp\dd \tilde x;
		\end{split}
	\end{align}
	in the regime $\alpha\geq 3$, we also use $(z_\eps)_\eps\subset H^1(0,L)$ with
	\begin{align}\label{def_z}
		z_\epsilon(x_1) =\displaystyle\frac{1}{\epsilon^{\alpha-1}}\int_\omega u_\epsilon\cdot e_1 - x_1 \dd \tilde x\quad \text{for $x_1\in (0, L)$},
	\end{align}
	which represent (appropriately scaled) versions of averaged length changes perpendicular and in-line with the midfiber, as well as torsion effects, respectively
	
 	Next, we introduce the limit energies in dependence of $\alpha$. For the scaling regime $\alpha\in (2,3)$, let $\Ical\ui{\alpha}: H^2(0,L;\R^2) \times H^1(0,L) \to [0,\infty)$ be given by
	\begin{align}\label{I^2-3}
	 \Ical\ui{\alpha}(v,w)= \frac{1}{2}\int_0^L \Qld(B'(x_1), 0) \dd x_1, 
	\end{align}
	where $\Qld(\cdot, 0)$ is the quadratic form in \eqref{Qast} (see also Remark~\ref{rem:splitting}) and $B\in H^1(0,L;\R^{3\times 3}_{\rm skew})$ is defined as
	\begin{align}\label{def_B}
		B = \begin{pmatrix}
		0 & -v_1' & -v_2' \\ v_1' & 0 & -w \\ v_2' & w & 0
		\end{pmatrix}. 
	\end{align}
	In the von K\'arm\'an-type regime $\alpha=3$ and for $\alpha>3$, we define $\Ical\ui{\alpha}: H^2(0,L;\R^2)\times H^1(0,L)\times H^1(0,L)\to [0,\infty)$ via
	\begin{align}\label{I^3}
		\begin{split}
				\Ical\ui{\alpha}(v,w,z)=\frac{1}{2} \int_0^L \Qld\big(B'(x_1), s\ui{\alpha}(x_1)\big)\dd x_1;
		\end{split}
	\end{align}here, the stored energy density $\Qld$ results from the constrained variational problem defined in~\eqref{Qast}, $B$ is as in~\eqref{def_B} and $s\ui{\alpha}\in L^2(0, L)$ is given by
\begin{align}\label{def_Salpha}	
s\ui{\alpha}=\begin{cases} z'+\frac{1}{2}|v'|^2 &\text{for $\alpha=3,$} \\ z' & \text{for $\alpha>3.$}\end{cases}
\end{align}

With these definitions at hand, we can formulate the following $\Gamma$-convergence result.

	\begin{theorem}[\boldmath{$\Gamma$}-limit for \boldmath{$\alpha>2$}]\label{theo:rods>2}
		 Let $\Ical_\eps^{(\alpha)}$ for $\eps>0$ be the functional introduced in~\eqref{Ical_eps^alpha} with $\alpha>2$ and let $\Ical\ui{\alpha}$ as in~\eqref{I^2-3} and~\eqref{I^3}, respectively.

		\textit{i) (Compactness)} 
		For every sequence $(\bar u_\eps)_\eps\subset H^1(\Omega;\R^3)$ with $\sup_{\eps>0} \Ical_\eps\ui{\alpha}(\bar u_\eps) <\infty$ 
	there exist sequences of translations $(\bar d_\eps)_\eps\subset\R^2$, rotations $(\bar R_\eps)_\eps\in\SO(3)$, and $\bar R\in\SO(3)$ with $\bar R_\eps \to \bar R$,
		as well as $v\in H^2(0,L;\R^2)$ and $w\in H^1(0,L)$ such that, with $u_\eps := \bar R_\eps\bar u_\eps - \bar d_\eps$, the following convergences hold up to the selection of subsequences:
		\begin{align}\label{conv_2-3}
			\begin{split}
				v_\eps \to v &\text{ in } H^1(0,L;\R^2),\\
				w_\eps \weakly w &\text{ in } H^1(0,L),\\
				\tfrac{1}{\eps^{\alpha-2}}(\rn{u} - \Id) \to B &\text{ in } L^2(\Omega;\R^{3\times 3});
			\end{split}
		\end{align}
recall the definitions of $v_\eps$, $w_\eps$ and $B$ in~\eqref{def_v,w} and \eqref{def_B}, respectively.
 Additionally, if $\alpha\geq 3$, there exists $z\in H^1(0,L)$ such that $(z_\eps)_\eps\subset H^1(0,L)$ as in \eqref{def_z} fulfills
		\begin{align}\label{conv_3}
			z_\eps\weakly z \text{ in } H^1(0,L).
		\end{align}	
			
		\textit{ii) (Variational limit)}  If $\alpha\in (2,3)$, the sequence $(\Ical_\eps\ui{\alpha})_\eps$ $\Gamma$-converges to $\Ical\ui{\alpha}$ for $\eps\to 0$ regarding the convergence \eqref{conv_2-3}. For $\alpha\geq 3$, $\Ical\ui{\alpha}$ is the $\Gamma$-limit of $(\Ical_\eps\ui{\alpha})_\eps$  with respect to the convergence \eqref{conv_2-3} and \eqref{conv_3}.
	\end{theorem}

	\begin{proof}
		Ad $i)$. Since any sequence $(\bar u_\eps)_\eps$ with uniformly bounded energy satisfies 
		\begin{align*}
			\frac{1}{\eps^{2\alpha -2}}\int_\Omega W_0(\rn{\bar u}) \dd x \leq \frac{1}{\eps^{2\alpha -2}} \int_\Omega W(\rn{\bar u}) \dd x = \Ical\ui{\alpha}_\eps(\bar u_\eps) \leq C
		\end{align*}
		for a constant $C>0$, and $W_0$ satisfies (H2), the statement follows directly from the literature on the compressible case. The compactness result for the von Karman-type case $\alpha=3$ was first proven in~\cite[Theorem 2.2]{MoM04}, for the remaining $\alpha>2$, we refer to~\cite[Theorem 3.3]{Sca09}, where all scaling regimes are covered in the more general context of curved rods.

	Ad $ii)$.  As pointed out in the introductory Section~\ref{subsec:approach}, the $\Gamma$-limits $\Ical_k\ui{\alpha}$ of $(\Ical\ui{\alpha}_{k,\eps})_\eps$ as in \eqref{penalty_energy} in the unconstrained setting provide lower bounds for the incompressible limit energy, which implies that
		\begin{align*}
			\Ical\ui{\alpha} \geq \sup_{k\in\N} \Ical\ui{\alpha}_k 
		\end{align*}
		with
		\begin{align*}
			\begin{cases} 
				\Ical\ui{\alpha}_k(v,w) = \displaystyle \frac{1}{2}\displaystyle\int_0^L \Qld_k(B'(x_1), 0) \dd x_1,&\text{ if } \alpha \in (2,3),\\
				\Ical\ui{\alpha}_k(v,w,z) =\displaystyle \frac{1}{2}\displaystyle\int_0^L \Qld_k\big(B'(x_1), s\ui{\alpha}(x_1)\big)\dd x_1,&\text{ if } \alpha \geq 3,
			\end{cases}
		\end{align*}	
		cf.~\eqref{def_B} and~\eqref{def_Salpha}.	 
	The sought liminf-inequality follows with the help of Corollary~\ref{cor:convergence_Q}.
		
	For easier reading, we copy the structure of the proof of Theorem~\ref{theo:rods=2} and subdivide the arguments for the upper bound in three steps.
		
		\textit{Step 1: Recovering smooth limit functions.}
		Let $v\in C^3([0,L];\R^2)$, $w\in C^2([0,L])$, and $B\in C^2([0,L];\R^{3\times 3}_{\skw})$ as in \eqref{def_B}, and  if $\alpha\geq 3$, let also $z\in C^2([0,L])$. 
	We choose $\beta \in C^2([0,L]\times \R^2;\R^3)$ such that
		\begin{align}\label{vanishing_trace}
			\begin{cases}
				\tr \big(B'(x_2e_2+x_3e_3)|\tnabla \beta\big) = 0,&\text{ if } \alpha\in(2,3),\\
				\tr \big(B'(x_2e_2+x_3e_3) + s\ui{\alpha} e_1 | \tnabla \beta\big)= 0&\text{ if } \alpha\geq 3.
			\end{cases}
		\end{align}
Furthermore, let $Q_L\subset Q_L'$ be cubes as in 
	Lemma \ref{lem:reparam_det=1} such that $Q_L$ contains $\Omega$. 

The basis for our construction of locally volume-preserving approximations $(u_\eps)_\eps$ are the recovery sequences from the literature on the compressible cases~\cite{MoM03, Sca09}. If $\alpha\in (2,3)$, we set	
\begin{align*}
	y_\eps(x) = \displaystyle \int_0^{x_1} R_\epsilon(s) e_1 \dd s  + \eps R_\epsilon(x_1)(x_2e_2 + x_3e_3) + \epsilon^{\alpha} \beta(x)
\end{align*}
for $x\in Q_L'$, where $R_\eps$ is the $\SO(3)$-valued matrix exponential $R_\eps := \exp(\eps^{\alpha-2}B)$ with $B$ as in~\eqref{def_B}, cf.~\cite[Theorem 5.2, (5.24)]{Sca09}.
For $\alpha\geq 3$, consider
\begin{align*}
	\displaystyle y_\eps(x)= x_\eps + \eps^{\alpha-2} (0,v(x_1)) + \eps^{\alpha-1} B(x_1) (x_2e_2 + x_3e_3) + \eps^{\alpha-1} z(x_1)e_1 +\eps^\alpha \beta\ui{\alpha}(x)	
\end{align*}
for $x\in Q_L'$, where $x_\eps = (x_1,\eps x_2, \eps x_3)$ and
\begin{align*}
	\beta\ui{\alpha}(x) := \begin{cases}
					\beta(x) -  \frac{1}{2}\bigl(x_2\gamma(x_1) + x_3\bar \gamma(x_1)\bigr),  &\text{ if } \alpha = 3,\\
					\beta(x),&\text{ if } \alpha > 3,
				\end{cases}
\end{align*}
with $\gamma := 2w v_2' e_1 + (w^2+|v_1'|^2) e_2 +  v_1'v_2'e_3$ and $\bar \gamma := - 2w v_1' e_1 + v_1' v_2' e_2 +  (w^2+|v_2'|)e_3$; 
		for more details in the case $\alpha=3$, see \cite[Theorem~3.1 and (4.14), (4.15)]{MoM04}, \cite[Theorem~5.1]{Sca09}, and for $\alpha\geq 3$, \cite[Theorem~5.1]{Sca09}. 
	
By these constructions, $(y_\eps)_\eps$ satisfies the desired convergences \eqref{conv_2-3}, and if $\alpha\geq 3$ also \eqref{conv_3}. The specific structure of $y_\eps$ makes it immediate to see that
\begin{align*}
	\norm{\partial_3y_\eps}_{C^1(Q_L';\R^3)} = \Ocal(\eps).
\end{align*}
Moreover, \eqref{vanishing_trace} in conjunction with the computations in~\cite{MoM04, Sca09} shows that 	
	\begin{align*}
			\norm{\det \rn{y} - 1}_{C^1(Q_L')} = o(\eps^{\alpha-1});
		\end{align*}
let us remark that in the case $\alpha=3$, one even obtains that the deviation of $\det \rn{y}$ from $1$ behaves like $ \Ocal(\eps^3)$, but indeed, $o(\eps^2)$ is sufficient for our purposes.

Therefore, we can now apply Lemma \ref{lem:reparam_det=1} to find a sequence $(u_\eps)_\eps\subset C^1(\overline{\Omega};\R^3)$
		such that 
		\begin{align*}
		\det \rn{u} = 1\quad \text{ in $\overline{\Omega}$}
		\end{align*}
		and
		\begin{align*}
			\norm{u_\eps - y_\eps}_{C^1(\overline{\Omega};\R^3)} = o(\eps^\alpha).
		\end{align*}
		This yields in particular, that $(u_\eps)_\eps$ converges as in \eqref{conv_2-3}, and additionally, if $\alpha\geq 3$, that $(u_\eps)_\eps$ satisfies \eqref{conv_3}. 
	Analogously to~\cite[Theorem~5.1, Theorem~5.2]{Sca09}, we obtain that
		\begin{align*}
			W(\rn{u}) \leq\begin{cases}
							\displaystyle	\tfrac{1}{2} \eps^{2\alpha-2}Q\big(B'(x_2e_2+x_3e_3)| \tnabla \beta\big) + o(\epsilon^{2\alpha-2}),&\text{ if } \alpha\in(2,3),\\[0.3cm]
							\displaystyle	\tfrac{1}{2} \eps^{2\alpha-2}Q\big(B'(x_2e_2+x_3e_3) + s\ui{\alpha} e_1| \tnabla \beta\big) + o(\epsilon^{2\alpha-2}),&\text{ if } \alpha=3,
							\end{cases}
		\end{align*}				
	and thus,
		\begin{align}\label{limsup_alpha}
			\limsup_{\eps\to 0} \Ical_\eps\ui{\alpha}(u_\eps) = \begin{cases}
							\displaystyle	\frac{1}{2} \int_\Omega Q\big(B'(x_2e_2+x_3e_3)| \tnabla \beta\big) \dd x ,&\text{ if } \alpha\in(2,3),\\[0.3cm]
							\displaystyle	\frac{1}{2} \int_\Omega Q\big(B'(x_2e_2+x_3e_3) + s\ui{\alpha}e_1| \tnabla \beta\big) \dd x,&\text{ if } \alpha=3.
							\end{cases}
		\end{align}

		\textit{Step 2: Approximation and optimzation.} 
		Let $v\in H^2(0,L;\R^2)$, $w\in H^1(0,L)$, and, if $\alpha\geq 3$, $z\in H^1(0,L)$. Then, there are $(\hat v_\delta)_\delta\subset C^3([0,L];\R^2)$, $(\hat w_\delta)_\delta\subset C^2([0,L])$ such that
		\begin{align*}
			\begin{split}
				\hat v_\delta \to v &\text{ in } H^2(0,L;\R^2),\\
				\hat w_\delta \to w &\text{ in } H^1(0,L).		
			\end{split}	
		\end{align*}
		Furthermore, in the case $\alpha\geq 3$, let $(\hat z_\delta)_\delta\subset C^2([0,L])$ such that
		\begin{align*}
			\hat z_\delta \to z &\text{ in } H^1(0,L).
		\end{align*}
		
		We define $\beta\in L^2(0,L;H^1(\omega;\R^3))$ as follows: 
		if $\alpha\in (2,3)$ and $\alpha\geq 3$, then $\beta(x_1,\cdot)$ is the (unique) solution with vanishing mean value for the minimization problem defining 
		$\Qld(B'(x_1), 0)$ and $\Qld(B'(x_1), s\ui{\alpha}(x_1))$, respectively, cf.~also Lemma~\ref{lem:existence}.

		Now, we apply Corollary \ref{cor:div_ext_approx} with 
		\begin{align*}
			\rho(x) =  	\begin{cases} 
							B'(x_1)(x_2e_2+x_3e_3)\cdot e_1,&\text{ if } \alpha\in(2,3),\\
							B'(x_1)(x_2e_2+x_3e_3)\cdot e_1 +  s\ui{\alpha}(x_1)&\text{ if } \alpha \geq 3,
						\end{cases}
		\end{align*}
		and
		\begin{align*}
			\rho_\delta(x) =	\begin{cases} 
							B'_\delta(x_1)(x_2e_2+x_3e_3)\cdot e_1,&\text{ if } \alpha\in(2,3),\\
							B'_\delta(x_1)(x_2e_2+x_3e_3)\cdot e_1 + s\ui{\alpha}_\delta(x_1)&\text{ if } \alpha=3,
						\end{cases}
		\end{align*}
 where $B_\delta$ and $s\ui{\alpha}_\delta$ are given as in \eqref{def_B} and~\eqref{def_Salpha} with $v, w$ replaced by their approximations $\hat v_\delta, \hat w_\delta$. 
		This provides us with a sequence $(\beta_\delta)_\delta\subset C^2([0,L]\times \R^2;\R^3)$ such that
		\begin{align}\label{traceconstraint22}
			\begin{cases}
				\tr \big(B'_\delta(x_2e_2+x_3e_3)|\tnabla \beta_\delta\big) = 0,&\text{ if } \alpha\in(2,3),\\
				\tr \big(B'_\delta(x_2e_2+x_3e_3) + s\ui{\alpha}_\delta e_1 | \tnabla \beta_\delta\big)= 0&\text{ if } \alpha\geq 3,
			\end{cases}
		\end{align}
		and $\beta_\delta\to \beta$ in $L^2(0,L;H^1(\omega;\R^3))$. 

	\textit{Step 3: Diagonalization.} Exactly as in Step 3 of the proof of Theorem \ref{theo:rods=2}, we apply Step~1 for every $\delta$ with $v=\hat v_\delta$, $w=\hat w_\delta$, $z=\hat z_\delta$ if $\alpha\geq 3$, and the approximation of the optimal choice for $\beta$ from Step~2, i.e., $\beta = \beta_\delta$. This way, we obtain a sequences $(u_{\delta,\eps})_{\eps}\subset H^1(\Omega;\R^3)$  that converge in the sense of~\eqref{conv_2-3}, and~\eqref{conv_3} if $\alpha\geq 3$, and satisfy in view of~\eqref{limsup_alpha}, ~\eqref{Qast} and~\eqref{traceconstraint22},
		\begin{align*}
			\limsup_{\eps\to 0} \Ical_\eps\ui{\alpha}(u_{\delta, \eps}) = \begin{cases}
							\displaystyle	\frac{1}{2} \int_\Omega \Qld (B'_\delta, 0) \dd x ,&\text{ if } \alpha\in(2,3),\\[0.3cm]
							\displaystyle	\frac{1}{2} \int_\Omega \Qld (B'_\delta, s_\delta\ui{\alpha}) \dd x,&\text{ if } \alpha=3.
						\end{cases}
		\end{align*}
	To finalize the proof, it suffices to extract a suitable diagonal sequence $(u_\eps)_\eps$ according to Attouch's lemma. 
	\end{proof}
	
\subsection*{Acknowledgements} 
CK acknowledges partial financial support by the UU Westerdijk fellowship program.


\bibliographystyle{abbrv}
\bibliography{Dimensionreduction}

\begin{thebibliography}{10}

\bibitem{ABP91}
E.~Acerbi, G.~Buttazzo, and D.~Percivale.
\newblock A variational definition of the strain energy for an elastic string.
\newblock {\em J. Elasticity}, 25(2):137--148, 1991.

\bibitem{ACD07}
J.~Adams, S.~Conti, and A.~DeSimone.
\newblock Soft elasticity and microstructure in smectic {C} elastomers.
\newblock {\em Contin. Mech. Thermodyn.}, 18(6):319--334, 2007.

\bibitem{ABR19}
M.~Amabili, I.~D. Breslavsky, and J.~N. Reddy.
\newblock Nonlinear higher-order shell theory for incompressible biological
  hyperelastic materials.
\newblock {\em Comput. Methods Appl. Mech. Engrg.}, 346:841--861, 2019.

\bibitem{Att84}
H.~Attouch.
\newblock {\em Variational convergence for functions and operators}.
\newblock Applicable Mathematics Series. Pitman (Advanced Publishing Program),
  Boston, MA, 1984.

\bibitem{Bar19}
S.~Bartels.
\newblock Finite element simulation of nonlinear bending models for thin
  elastic rods and plates.
\newblock {\em Preprint: arXiv:1901.09835}.

\bibitem{BLS16}
K.~Bhattacharya, M.~Lewicka, and M.~Sch\"{a}ffner.
\newblock Plates with incompatible prestrain.
\newblock {\em Arch. Ration. Mech. Anal.}, 221(1):143--181, 2016.

\bibitem{Bra02}
A.~Braides.
\newblock {\em {$\Gamma$}-convergence for beginners}, volume~22 of {\em Oxford
  Lecture Series in Mathematics and its Applications}.
\newblock Oxford University Press, Oxford, 2002.

\bibitem{BPV17}
M.~Bukal, M.~Pawelczyk, and I.~Vel\v{c}i\'{c}.
\newblock Derivation of homogenized {E}uler-{L}agrange equations for von
  {K}\'{a}rm\'{a}n rods.
\newblock {\em J. Differential Equations}, 262(11):5565--5605, 2017.

\bibitem{ChL13}
M.~Chermisi and H.~Li.
\newblock The von {K}\'arm\'an theory for incompressible elastic shells.
\newblock {\em Calc. Var. Partial Differential Equations}, 48(1-2):185--209,
  2013.

\bibitem{Cia97}
P.~Ciarlet.
\newblock {\em Mathematical Elasticity: Theory of Plates}.
\newblock Developments in Aquaculture and Fisheries Science. North-Holland,
  1997.

\bibitem{CRS17b}
M.~Cicalese, M.~Ruf, and F.~Solombrino.
\newblock Hemihelical local minimizers in prestrained elastic bi-strips.
\newblock {\em Z. Angew. Math. Phys.}, 68(6):Art. 122, 18, 2017.

\bibitem{CRS17a}
M.~Cicalese, M.~Ruf, and F.~Solombrino.
\newblock On global and local minimizers of prestrained thin elastic rods.
\newblock {\em Calc. Var. Partial Differential Equations}, 56(4):Art. 115, 34,
  2017.

\bibitem{CoD06}
S.~Conti and G.~Dolzmann.
\newblock Derivation of elastic theories for thin sheets and the constraint of
  incompressibility.
\newblock In {\em Analysis, modeling and simulation of multiscale problems},
  pages 225--247. Springer, Berlin, 2006.

\bibitem{CoD09}
S.~Conti and G.~Dolzmann.
\newblock {$\Gamma$}-convergence for incompressible elastic plates.
\newblock {\em Calc. Var. Partial Differential Equations}, 34(4):531--551,
  2009.

\bibitem{CoD15}
S.~Conti and G.~Dolzmann.
\newblock On the theory of relaxation in nonlinear elasticity with constraints
  on the determinant.
\newblock {\em Arch. Ration. Mech. Anal.}, 217(2):413--437, 2015.

\bibitem{Dal93}
G.~Dal~Maso.
\newblock {\em An introduction to {$\Gamma$}-convergence}, volume~8 of {\em
  Progress in Nonlinear Differential Equations and their Applications}.
\newblock Birkh\"{a}user Boston, Inc., Boston, MA, 1993.

\bibitem{DaM12}
E.~Davoli and M.~G. Mora.
\newblock Convergence of equilibria of thin elastic rods under physical growth
  conditions for the energy density.
\newblock {\em Proc. Roy. Soc. Edinburgh Sect. A}, 142(3):501--524, 2012.

\bibitem{DeD02}
A.~DeSimone and G.~Dolzmann.
\newblock Macroscopic response of nematic elastomers via relaxation of a class
  of {$\rm SO(3)$}-invariant energies.
\newblock {\em Arch. Ration. Mech. Anal.}, 161(3):181--204, 2002.

\bibitem{EnK19}
D.~Engl and C.~Kreisbeck.
\newblock Asymptotic variational analysis of incompressible elastic strings.
\newblock {\em Preprint: arXiv:1909.07901}.

\bibitem{Fri47}
K.~O. Friedrichs.
\newblock On the boundary-value problems of the theory of elasticity and
  {K}orn's inequality.
\newblock {\em Ann. of Math. (2)}, 48:441--471, 1947.

\bibitem{FJM02}
G.~Friesecke, R.~D. James, and S.~M\"uller.
\newblock A theorem on geometric rigidity and the derivation of nonlinear plate
  theory from three-dimensional elasticity.
\newblock {\em Comm. Pure Appl. Math.}, 55(11):1461--1506, 2002.

\bibitem{FJM06}
G.~Friesecke, R.~D. James, and S.~M\"uller.
\newblock A hierarchy of plate models derived from nonlinear elasticity by
  gamma-convergence.
\newblock {\em Arch. Ration. Mech. Anal.}, 180(2):183--236, 2006.

\bibitem{GiR86}
V.~Girault and P.-A. Raviart.
\newblock {\em Finite element methods for {N}avier-{S}tokes equations},
  volume~5 of {\em Springer Series in Computational Mathematics}.
\newblock Springer-Verlag, Berlin, 1986.
\newblock Theory and algorithms.

\bibitem{Hor11}
P.~Hornung.
\newblock Invertibility and non-invertibility in thin elastic structures.
\newblock {\em Arch. Ration. Mech. Anal.}, 199(2):353--368, 2011.

\bibitem{HNV14}
P.~Hornung, S.~Neukamm, and I.~Vel\v{c}i\'{c}.
\newblock Derivation of a homogenized nonlinear plate theory from 3d
  elasticity.
\newblock {\em Calc. Var. Partial Differential Equations}, 51(3-4):677--699,
  2014.

\bibitem{KMPT00}
T.~Kato, M.~Mitrea, G.~Ponce, and M.~Taylor.
\newblock Extension and representation of divergence-free vector fields on
  bounded domains.
\newblock {\em Math. Res. Lett.}, 7(5-6):643--650, 2000.

\bibitem{KoO18}
R.~V. Kohn and E.~O'Brien.
\newblock On the bending and twisting of rods with misfit.
\newblock {\em J. Elasticity}, 130(1):115--143, 2018.

\bibitem{KoO88}
V.~A. Kondratev and O.~A. Oleinik.
\newblock Boundary-value problems for the system of elasticity theory in
  unbounded domains. korn's inequalities.
\newblock {\em Russian Mathematical Surveys}, 43(5):65--119, oct 1988.

\bibitem{LeR95}
H.~Le~Dret and A.~Raoult.
\newblock The nonlinear membrane model as variational limit of nonlinear
  three-dimensional elasticity.
\newblock {\em J. Math. Pures Appl. (9)}, 74(6):549--578, 1995.

\bibitem{LeR96}
H.~Le~Dret and A.~Raoult.
\newblock The membrane shell model in nonlinear elasticity: A variational
  asymptotic derivation.
\newblock {\em Journal of Nonlinear Science}, 6(1):59--84, Jan 1996.

\bibitem{LeL15}
M.~Lewicka and H.~Li.
\newblock Convergence of equilibria for incompressible elastic plates in the
  von {K}\'{a}rm\'{a}n regime.
\newblock {\em Commun. Pure Appl. Anal.}, 14(1):143--166, 2015.

\bibitem{MoM03}
M.~G. Mora and S.~M\"uller.
\newblock Derivation of the nonlinear bending-torsion theory for inextensible
  rods by {$\Gamma$}-convergence.
\newblock {\em Calc. Var. Partial Differential Equations}, 18(3):287--305,
  2003.

\bibitem{MoM04}
M.~G. Mora and S.~M\"uller.
\newblock A nonlinear model for inextensible rods as a low energy
  {$\Gamma$}-limit of three-dimensional nonlinear elasticity.
\newblock {\em Ann. Inst. H. Poincar\'e Anal. Non Lin\'eaire}, 21(3):271--293,
  2004.

\bibitem{MoM08}
M.~G. Mora and S.~M\"{u}ller.
\newblock Convergence of equilibria of three-dimensional thin elastic beams.
\newblock {\em Proc. Roy. Soc. Edinburgh Sect. A}, 138(4):873--896, 2008.

\bibitem{MoS12}
M.~G. Mora and L.~Scardia.
\newblock Convergence of equilibria of thin elastic plates under physical
  growth conditions for the energy density.
\newblock {\em J. Differential Equations}, 252(1):35--55, 2012.

\bibitem{MuP08}
S.~M\"{u}ller and M.~R. Pakzad.
\newblock Convergence of equilibria of thin elastic plates---the von
  {K}\'{a}rm\'{a}n case.
\newblock {\em Comm. Partial Differential Equations}, 33(4-6):1018--1032, 2008.

\bibitem{Neu12}
S.~Neukamm.
\newblock Rigorous derivation of a homogenized bending-torsion theory for
  inextensible rods from three-dimensional elasticity.
\newblock {\em Arch. Ration. Mech. Anal.}, 206(2):645--706, 2012.

\bibitem{NeV13}
S.~Neukamm and I.~Vel\v{c}i\'{c}.
\newblock Derivation of a homogenized von-{K}\'{a}rm\'{a}n plate theory from
  3{D} nonlinear elasticity.
\newblock {\em Math. Models Methods Appl. Sci.}, 23(14):2701--2748, 2013.

\bibitem{Ogd72}
R.~W. Ogden.
\newblock Large deformation isotropic elasticity - on the correlation of theory
  and experiment for incompressible rubberlike solids.
\newblock {\em Proceedings of the Royal Society of London. Series A,
  Mathematical and Physical Sciences}, 326(1567):565--584, 1972.

\bibitem{OlR17}
H.~Olbermann and E.~Runa.
\newblock Interpenetration of matter in plate theories obtained as
  {$\Gamma$}-limits.
\newblock {\em ESAIM Control Optim. Calc. Var.}, 23(1):119--136, 2017.

\bibitem{Pan03}
O.~Pantz.
\newblock On the justification of the nonlinear inextensional plate model.
\newblock {\em Arch. Ration. Mech. Anal.}, 167(3):179--209, 2003.

\bibitem{Pey15}
J.~Peypouquet.
\newblock {\em Convex optimization in normed spaces}.
\newblock Springer Briefs in Optimization. Springer, Cham, 2015.
\newblock Theory, methods and examples, With a foreword by Hedy Attouch.

\bibitem{Sca09}
L.~Scardia.
\newblock Asymptotic models for curved rods derived from nonlinear elasticity
  by {$\Gamma$}-convergence.
\newblock {\em Proc. Roy. Soc. Edinburgh Sect. A}, 139(5):1037--1070, 2009.

\bibitem{TrV96}
L.~Trabucho and J.~M. Via\~{n}o.
\newblock Mathematical modelling of rods.
\newblock In {\em Handbook of numerical analysis, {V}ol. {IV}}, Handb. Numer.
  Anal., IV, pages 487--974. North-Holland, Amsterdam, 1996.

\end{thebibliography}
\end{document}